\documentclass[11pt,a4paper]{article}
\usepackage{amsmath,amssymb,amsthm,enumerate,hyperref}
\usepackage[dvips]{color,graphicx}

\def\C{{\mathcal{C}}}
\def\E{{\mathcal{E}}}

\def\N{\mathbb{N}}
\def\R{\mathbb{R}}

\def\emp{\emptyset}

\def\E{{\cal E}}

\def\N{{\cal N}}

\def\O{{\cal O}}

\def\ox{\overline{x}}
\def\oy{\overline{y}}
\def\oz{\overline{z}}

\def\cl{{\rm cl}\,}

\def\disp{\displaystyle}

\def\tto{\;{\lower 1pt\hbox{$\rightarrow$}}\kern-10pt
\hbox{\raise 2pt\hbox{$\rightarrow$}}\;}
\def\Hat{\widehat}

\def\Bar{\overline}
\def\ra{\rangle}
\def\la{\langle}

\def\R{\mathbb{R}}
\def\N{I\!\!N}
\def\ox{\bar{x}}
\def\oy{\bar{y}}
\def\oz{\bar{z}}

\def\gph{\mbox{\rm gph}\,}

\def\cl{{\rm cl}}

\def\substack#1#2{{\scriptstyle{#1}\atop\scriptstyle{#2}}}

\def\O{\Omega}
\def\ph{\varphi}
\def\emp{\emptyset}
\def\st{\stackrel}
\def\oR{\Bar{\R}}

\def\al{\alpha}

\def\N{I\!\!N}
\def\th{\theta}
\def\vt{\vartheta}
\def\beq{\begin{equation}}
\def\eeq{\end{equation}}
\def\sce{\setcounter{equation}{0}}
\begin{document}
\newtheorem{Theorem}{Theorem}[section]
\newtheorem{Corollary}[Theorem]{Corollary}
\newtheorem{Lemma}[Theorem]{Lemma}
\newtheorem{Proposition}[Theorem]{Proposition}
\newtheorem{Assumption}[Theorem]{Assumption}
\newtheorem{Definition}[Theorem]{Definition}
\newtheorem{exmp}{Example}[section]
\newtheorem{Fact}[Theorem]{Fact}
\newtheorem{Remark}[Theorem]{Remark}
\numberwithin{equation}{section}
\allowdisplaybreaks
\begin{center}
{\bf SUBDIFFERENTIALS OF NONCONVEX INTEGRAL FUNCTIONALS IN BANACH SPACES\\ WITH APPLICATIONS TO\\ STOCHASTIC DYNAMIC PROGRAMMING}\\[3ex]
BORIS S. MORDUKHOVICH\footnote{Department of Mathematics, Wayne State University, Detroit, Michigan 48202, USA (boris@math.wayne.edu). Research of this author was partly supported by the USA National Science Foundation under grants DMS-1007132 and DMS-1512846 and by the USA Air Force Office of Scientific Research under grant No.\,15RT0462.} and NOBUSUMI SAGARA\footnote{Department of Economics, Hosei University, Tokyo 194-0298, Japan (nsagara@hosei.ac.jp). Research of this author was partly supported by JSPS KAKENHI Grant Number 26380246 from the Ministry of Education, Culture, Sports, Science and Technology, Japan.}\\[2ex]
{\bf Dedicated to Nino Maugeri with great respect}
\end{center}

\small{\bf Abstract.} The paper concerns the investigation of nonconvex and nondifferentiable integral functionals on general Banach spaces, which may not be reflexive and/or separable. Considering two major subdifferentials of variational analysis, we derive nonsmooth versions of the Leibniz rule on subdifferentiation under the integral sign, where the integral of the subdifferential set-valued mappings generated by Lipschitzian integrands is understood in the Gelfand sense. Besides examining integration over complete measure spaces and also over those with nonatomic measures, our special attention is drawn to a stronger version of measure nonatomicity, known as saturation, to invoke the recent results of the Lyapunov convexity theorem type for the Gelfand integral of the subdifferential mappings. The main results are applied to the subdifferential study of the optimal value functions and deriving the corresponding necessary optimality conditions in nonconvex problems of stochastic dynamic programming with discrete time on the infinite horizon.\vspace*{0.05in}

\vspace*{0.05in}


\normalsize

\section{Introduction}\sce

This present study belongs to the area of infinite-dimensional variational analysis and its applications to which Professor Antonino Maugeri has made well-recognized contributions; see, e.g., \cite{mau01} and the references therein.

The main object of our investigation in this paper is the class of {\em integral functionals} of the type
\begin{eqnarray}\label{int-f}
I_\varphi(x):=\int\varphi(x,\omega)d\mu,\quad x\in X,
\end{eqnarray}
generated by some integrand $\ph\colon X\times\O\to\oR:=(-\infty,\infty]$, which is defined on a Banach space $X$ and a {\em finite measure} space $\O$. Our major goal is to establish new results of {\em subdifferential calculus} for \eqref{int-f} related to {\em nonconvex subdifferential} versions of the classical {\em Leibniz rule} on differentiation under the integral sign. Subdifferentiation of integral functionals of type \eqref{int-f} and appropriate versions of the Leibniz rule were developed in \cite{il72,ro71,va75} for convex integrands and then partially extended in \cite{ch09,cl83,gi08,jt11,mo06,pe11} to some nonconvex settings. In particular, the inclusion
\begin{eqnarray}\label{lem}
\partial\Big(\int^{1}_{0}\ph(\cdot,t)\,dt\Big)(\ox)\subset\cl\int^{1}_0\partial\ph(\ox,t)\,dt
\end{eqnarray}
was established in \cite[Lemma~6.18]{mo06} for the basic/limiting subdifferential by Mordukhovich \cite{mo06a}, where $\ph\colon X\times[0,1]\to\R$ is Lebesgue measurable in $t$ and locally Lipschitzian in $x$ around $\ox$ with a summable modulus on $[0,1]$, where $X$ is reflexive and separable, where the closure ``cl" is taken in the norm topology of $X^*$, and where the integral is understood in the Bochner sense. The obtained formula \eqref{lem} was applied in \cite[Chapter~6]{mo06} to the derivation of necessary optimality conditions of the extended Euler-Lagrange type for generalized Bolza problems governed by evolution inclusions.\vspace*{0.03in}

The original motivation for the current paper was to develop appropriate versions of \eqref{lem} in order to achieve the following goals:

{\bf(a)} Get an extension of \eqref{lem} to more {\em general classes} of {\em measure spaces}.

{\bf(b)} Obtain conditions allowing us to {\em dismiss the closure operation} in the suitable counterparts of \eqref{lem}.

{\bf(c)} {\em Avoid or relax} the {\em reflexivity and separability} assumptions on the Banach space $X$ for the validity of \eqref{lem} and its extensions.\vspace*{0.03in}

In the rest of this paper the reader can find a resolution of the issues listed in (a)--(c) together with additional results in these directions. The crucial ingredients of the progress achieved include the following:\vspace*{0.03in}

{\bf(i)} The usage of the {\em Gelfand integral} instead of the Bochner one in the case of mappings on complete measure spaces with values in a space, which is topologically {\em dual} to an {\em arbitrary Banach} space. Note that both these integrals agree in the setting of \cite[Lemma~6.18]{mo06}; in fact, in any separable Asplund space $X$. Our achievements in this direction are based on a quite recent progress in studying the Gelfand integrand of set-valued mappings made by Cascales et al. in \cite{ckr10,ckr11}; see Section~2 for more details.

{\bf(ii)} The usage of the measure {\em nonatomicity} and the {\em Lyapunov convexity theorem}, which allows us to derive an extension of \eqref{lem}  to general measure spaces $\O$ enjoying this property and (nonreflexive and nonseparable) {\em Asplund spaces} $X$ with the replacing the norm closure ``cl" on the right-hand side of \eqref{lem} by the weak$^*$ closure; the latter in not needed if $X$ is reflexive.

{\bf(iii)} The usage of {\em saturation} (or ``super-atomlessness") property of measure spaces, which is equivalent to the enhanced version of the Lyapunov convexity theorem for the Gelfand integral of set-valued mapping {\em without the closure operation} on the right-hand side of \eqref{lem}; see \cite{ks15a,po08,sy08}.\vspace*{0.03in}

It is worth noting that the aforementioned results obtained in the vein of \eqref{lem} are new even in the case of the Bochner integral in reflexive and separable Banach spaces $X$. Observe also that, in the process of implementing our approach, we establish similar inclusions for the Gelfand integral (more precisely, its $w^*$-integral version for convex-valued multifunctions) of the {\em generalized gradient} by Clarke \cite{cl83}, which is the convexification of the limiting subdifferential in the case of Asplund spaces. Moreover, it is shown below that the Gelfand integral counterpart of \eqref{lem} holds for the generalized gradient in any Banach space with omitting the closure operation on the right-hand side in the case of an arbitrary complete measure space $\O$, even without imposing the saturation requirement on the measure.

The obtained subdifferential formulas for integral functionals are further applied to problems of {\em stochastic dynamic programming} (DP) for infinite-dimensional discrete-time systems on the infinite horizon, which are important, e.g., in macroeconomic modeling. We first derive the Bellman equation for the intrinsically nonsmooth random {\em optimal value function} in such systems, then justify its local Lipschitz continuity under natural assumptions and evaluate its subdifferentials. These results readily imply necessary optimality conditions in the stochastic dynamic programs under consideration.\vspace*{0.03in}

The rest of the paper is organized as follows. Section~2 is devoted to the study of the Gelfand integral and its $w^*$-integral modification for set-valued mappings/multifunctions with values in duals to Banach spaces. We provide a brief overview of some basic facts needed in the sequel and obtain a new result on the convexity of the Gelfand integral of set-valued mappings defined on nonatomic measure spaces with no separability assumption on the Banach space in question.

Section~3 contains our major results on subdifferentiation of integral functionals of type \eqref{int-f}. First we derive a generalized gradient version of the Leibniz formula in the case of integrands $\ph(x,\omega)$ defined on the product of an arbitrary Banach space and a complete measure space, provided that the generalized gradient of $\ph$ in $x$ is $w^*$-scalarly measurable. Effective sufficient conditions for the latter properties are discussed in detail. Then we proceed with several versions of the Leibniz rule in the case of Gelfand integral functionals associate with the limiting subdifferential. In this way we resolve the issues discussed above in (a)--(c).

The last two sections concern applications of the obtained subdifferential results for integral functionals to problems of stochastic DP in Banach spaces. In Section~4 we describe a general stochastic DP problem for discrete-time systems on the infinite horizon and show that its optimal value function is a unique solution to the corresponding Bellman equation in an integral form with respect to a transition probability measure induced by the Markov decision process. Proving the local Lipschitz continuity of the value function under appropriate assumptions, we employ the results of Section~3 to evaluate the corresponding subdifferentials of the value function and derive necessary optimality conditions in the stochastic DP problem.\vspace*{0.03in}

Throughout the paper we use the standard terminology and notation recalled in the places they first appear in the text below.

\section{Gelfand Integrals of Set-Valued Mappings}

We start this section by recalling the notion of the $w^*$-integral for a function on a measurable space with values in dual Banach spaces. Let $(\Omega,\Sigma,\mu)$ be a complete finite measure space, and let $X$ be a real Banach space with the dual system $\langle X,X^*\rangle$, where $X^*$ is the norm/topological dual of $X$, while $w^*$ indicates the usage of the weak$^*$ topology on $X^*$ in what follows. A function $f:\Omega\to X^*$ is \textit{$\mathit{w}^*$-\hspace{0pt}scalarly measurable} if the scalar function $\langle f(\cdot),x\rangle$ is measurable for every $x\in X$. We say that a $\mathit{w}^*$-scalarly measurable function is \textit{$\mathit{w}^*$-scalarly integrable} if $\langle f(\cdot),x \rangle$ is integrable for every $x\in X$. Further, we say that a $\mathit{w}^*$-\hspace{0pt}scalarly measurable function $f$ is {\em $w^*$-integrable} (or {\em Gelfand integrable}) over a given set $A\in\Sigma$ if there exists $x^*_A\in X^*$ such that $\langle x^*_A,x \rangle=\int_A\langle f(\omega),x \rangle d\mu$ for every $x\in X$. The element $x^*_A$, which is unique by the separation theorem, is called the {\em Gelfand} or {\rm $w^*$-integral} of $f$ over $A$ and is denoted by $\int_Afd\mu$ . If $A=\O$, we omit indicating the set in the integral sign. Note that every $\mathit{w}^*$-scalarly integrable function is $\mathit{w}^*$-integrable as shown in \cite[Theorem 11.52]{ab06}. While we do not distinguish between the Gelfand integral and the $w^*$-integral of single-valued mappings, we do it for set-valued mappings/multifunctions; see below.

\begin{Definition}[\bf Gelfand integral of general set-valued mappings]\label{gel}Let $(\Omega,\Sigma,\mu)$ be a complete finite measure space, and let $X$ be a Banach space. Given a set-valued mapping $\Gamma\colon\O\rightrightarrows X^*$, we denote by ${\cal S}^1_\Gamma$ the collection of all the ${w}^*$-\hspace{0pt}integrable selectors of $\Gamma$ on $\O$, i.e., such $w^*$-integrable functions $f\colon\O\to X^*$ that $f(\omega)\in\Gamma(\omega)$ for a.e.\ $\omega\in\O$. Then the {\sc Gelfand integral} of the multifunction $\Gamma$ over the measure space $\O$ is defined by
\begin{equation}\label{eq0}
I(\Gamma)=\int\Gamma\,d\mu:=\left\{\int fd\mu\mid f\in\mathcal{S}^1_\Gamma\right\}.
\end{equation}
\end{Definition}

The first question arising in the study of the Gelfand integral \eqref{eq0} is about {\em nonemptiness} of the set $I(\Gamma)$. The answer is affirmative in a rather broad setting due to a nice selection theorem from \cite{ckr11}. Although our standing framework in this paper is the general class of Banach spaces, this result and some other ones require the Asplund property of the space in question. Recall that a Banach space $X$ is {\em Asplund} if every convex continuous function $\ph\colon U\to\R$  defined on an open convex set $U\subset X$ is Fr\'echet differentiable on a dense subset of $U$. This class of Banach spaces is sufficiently large including, in particular, any space with a Fr\'echet differentiable bump function (hence any space admitting an equivalent norm Fr\'echet differentiable off the origin, i.e., a {\em Fr\'echet smooth renorm}, and therefore every reflexive space), any space with a separable dual, and any space $X$ whose dual $X^*$ is {\em weakly compactly generated} meaning that there exists a weakly compact subset of $X^*$ whose linear span in norm sense. There are many useful characterizations of Asplund spaces; among the most remarkable ones we mention that $X$ is Asplund if and only if every closed separable subspace of $X^*$ has a separable dual. It is relevant to mention that any separable Asplund space admits a Fr\'echet smooth renorm. We refer the reader to the book \cite{ph93} and the bibliography therein for all these and other facts on Asplund spaces.

An appropriate definition of measurability for set-valued mappings is needed to be recalled in order to list conditions ensuring the nonemptiness of the Gelfand integral \eqref{eq0}. Denoting by
\begin{eqnarray}\label{supp}
s(x,C):=\sup_{x^*\in C}\langle x^*,x \rangle,\quad x\in X,
\end{eqnarray}
the {\em support function} of a nonempty set $C\subset X$, we say that a set-valued mapping $\Gamma\colon\O\rightrightarrows X^*$ is $\mathit{w}^*$-\hspace{0pt}\textit{scalarly measurable} if its support function $s(x,\Gamma)\colon\Omega\to\oR$ is measurable for every $x\in X$. Further, a multifunction $\Gamma$ is said to be \textit{integrably bounded} if there exists $\psi\in L^1(\mu)$ such that $\sup_{x^*\in\Gamma(\omega)}\|x^*\|\le\psi(\omega)$ for every $\omega\in\Omega$. Now we are ready to formulate the aforementioned result on nonemptiness of the Gelfand integral \eqref{eq0}.

\begin{Proposition}[\bf nonemptiness of the Gelfand integral]\label{thm:ckr2} Let $(\Omega,\Sigma,\mu)$ be a complete finite measure space, and let $X$ be an Asplund space. Assume that $\Gamma\colon\O\rightrightarrows X^*$ is an integrably bounded and $\mathit{w}^*$-\hspace{0pt}scalarly measurable multifunction with $\mathit{w}^*$-\hspace{0pt}closed values. Then the Gelfand integral \eqref{eq0} of the mapping $\Gamma$ is a nonempty subset of $X^*$.
\end{Proposition}
\begin{proof}
Since $\Gamma$ has $\mathit{w}^*$-compact values and its integral boundedness guarantees the $\mathit{w}^*$-integrability of $\mathit{w}^*$-\hspace{0pt}scalarly measurable selectors, the nonemptiness of the Gelfand integral follows from definition \eqref{eq0} and the measurable selection result from \cite[Corollary~3.11]{ckr11}, which ensures the existence of $\mathit{w}^*$-\hspace{0pt}scalarly measurable selectors of $\Gamma$ under the assumptions made.
\end{proof}

At this stage, it makes sense to compare the Gelfand and Bochner integrals of single-valued and set-valued mappings.

\begin{Remark}[\bf comparison with the Bochner integral]\label{rem1} {\rm Denote by $\mathrm{Borel}(X^*,\mathit{w}^*)$ the Borel $\sigma$-\hspace{0pt}algebra of $X^*$ generated by the weak$^*$ topology and by $\mathrm{Borel}(X^*,\|\cdot\|)$ the Borel $\sigma$-\hspace{0pt}algebra of $X^*$ generated by the dual norm. Recall that a function $f:\Omega\to X^*$ is \textit{strongly} (or \textit{Bochner}) \textit{measurable} if it is a pointwise $X^*$-norm limit of a sequence of $X^*$-\hspace{0pt}valued simple functions. It is well known that if $X$ is a separable Banach space, then $X^*$ is a locally convex Suslin space under the weak$^*$ topology and $f$ is $\mathit{w}^*$-\hspace{0pt}scalarly measurable if and only if it is $\mathrm{Borel}(X^*,\mathit{w}^*)$-\hspace{0pt}measurable; see \cite[Theorem~1]{th75} or \cite[Lemma 2]{va75}. Furthermore, if $X^*$ is separable, then $\mathrm{Borel}(X^*,\mathit{w}^*)$ agrees with $\mathrm{Borel}(X^*,\|\cdot\|)$ by \cite[Corollary~2 from Part~I, Ch.\,II]{sc73}. It implies, by taking into account that equality, the strong, $\mathrm{Borel}(X^*,\mathit{w}^*)$-\hspace{0pt}, and $\mathrm{Borel}(X^*,\|\cdot\|)$-\hspace{0pt}measurabilities coincide. Combining these facts allows us to deduce that a Gelfand integrable function $f$ satisfying $\int\|f(\omega)\|d\mu<\infty$ is also Bochner integrable if $X$ is Asplund and separable. Therefore, defining similarly to \eqref{eq0} the Bochner integral of a set-valued mapping $\Gamma\colon\O\rightrightarrows X^*$, we conclude that the {\em Gelfand and Bochner integrals} of $\Gamma$ {\em agree in separable Asplund spaces}.}
\end{Remark}

To proceed with studying the integral \eqref{eq0} of multifunctions $\Gamma\colon\O\rightrightarrows X^*$, recall that $\Gamma$ is \textit{$\mathit{w}^*$-\hspace{0pt}scalarly integrable} if its support function $s(x,\Gamma)$ is integrable for every $x\in X$. The next proposition is useful in what follows.

\begin{Proposition} [\bf supremum representation for the Gelfand integral]\label{lem04} Let $X$ be an Asplund space, $(\Omega,\Sigma,\mu)$ be a complete finite measure space, and $\Gamma\colon\Omega\rightrightarrows X^*$ be a Gelfand integrable, $\mathit{w}^*$-\hspace{0pt}scalarly integrable multifunction with $\mathit{w}^*$-\hspace{0pt}compact values. Then we have
\begin{eqnarray}\label{eq-sup}
\int s(x,\Gamma)d\mu=\sup_{f\in\mathcal{S}_\Gamma^1}\int\langle f(\omega),x \rangle d\mu\quad\text{for every $x\in X$}.
\end{eqnarray}
\end{Proposition}
\begin{proof}
For any $x\in X$ we define the multifunction $\Gamma_x\colon\Omega\rightrightarrows X^*$ with $\mathit{w}^*$-\hspace{0pt}compact values in $X^*$ by
$$
\Gamma_x(\omega):=\{x^*\in\Gamma(\omega)\mid\langle x^*,x \rangle=s(x,\Gamma(\omega))\}.
$$
As observed in \cite[Remark~4.4]{ckr11}, $\Gamma_x$ is $\mathit{w}^*$-\hspace{0pt}scalarly measurable, and hence it admits a $\mathit{w}^*$-\hspace{0pt}scalarly measurable selector $f_x$ by \cite[Corollary~3.11]{ckr11}. Thus
$$
s(x,\Gamma(\omega))=\langle f_x(\omega),x\rangle\;\mbox{ for a.e. }\;\omega\in\Omega.
$$
Integrating both sides of the above equality tells us that
$$
\int s(x,\Gamma)d\mu=\int\langle f_x(\omega),x \rangle d\mu\le\sup_{f\in\mathcal{S}_\Gamma^1}\int\langle f(\omega),x \rangle d\mu.
$$
On the other hand, it follows from \eqref{supp} that $s(x,\Gamma(\omega))\ge\langle x,f(\omega)\rangle$ for every $f\in\mathcal{S}_\Gamma^1$ a.e.\ $\omega\in\Omega$. Integrating both sides of this inequality yields
$$
\int s(x,\Gamma)d\mu\ge\sup_{f\in\mathcal{S}_\Gamma^1}\int\langle f(\omega),x \rangle d\mu,
$$
which readily verifies the claimed equality \eqref{eq-sup}.
\end{proof}

Next we describe an {\em alternative approach} to constructing a Gelfand-type integral of {\em convex}-valued multifunctions, which agrees with the Gelfand integral from Definition~\ref{gel} in the case of separable spaces $X$ while it may be a bit different from \eqref{eq0} in the absence of separability. To proceed, denote by $\overline{\mathrm{co}}^{\mathit{\,w}^*}\Gamma$ the multifunction defined by taking the {\em $\mathit{w}^*$-\hspace{0pt}closure} of the {\em convex hull} of $\Gamma(\omega)$ for every $\omega\in \Omega$. It follows from \eqref{supp} that $s(x,\Gamma(\omega))=s(x,\overline{\mathrm{co}}^{\mathit{\,w}^*}\Gamma(\omega))$ for every $x\in X$ and $\omega\in \Omega$, and that $\Gamma$ is $\mathit{w}^*$-\hspace{0pt}scalarly measurable (resp.\ $\mathit{w}^*$-\hspace{0pt}scalarly integrable) if and only if so is $\overline{\mathrm{co}}^{\mathit{\,w}^*}\Gamma$. We say that $f:\Omega\to X^*$ is a \textit{$\mathit{w}^*$-\hspace{0pt}almost selector} of $\Gamma$ if, for every $x\in X$
$$
\langle f(\omega),x \rangle\le s(x,\Gamma(\omega))\;\mbox{ for a.e.}\;\omega\in\Omega,
$$
where the corresponding $\mu$-\hspace{0pt}null set may depend on $x$. Similarly to the set ${\cal S}^1_\Gamma$ in Definition~\ref{gel}, denote by
$\mathcal{S}^{1,\mathit{w}^*}_\Gamma$ the collection of $\mathit{w}^*$-integrable, $\mathit{w}^*$-\hspace{0pt}almost selectors of a multifunction $\Gamma:\Omega\rightrightarrows X^*$. We clearly have the inclusion $\mathcal{S}^1_\Gamma\subset\mathcal{S}^{1,\mathit{w}^*}_\Gamma$. If furthermore $X$ is separable and $\Gamma$ has $\mathit{w}^*$-\hspace{0pt}compact and convex values, then $\mathcal{S}^1_\Gamma=\mathcal{S}^{1,\mathit{w}^*}_\Gamma$ as shown in \cite[Proposition 2.3]{ckr11}.

\begin{Definition}[\bf $w^*$-integral of convex-valued multifunctions]\label{gel-convex} Let $(\Omega,\Sigma,\mu)$ be a complete finite measure space, and let $X$ be Banach. We say that a $\mathit{w}^*$-\hspace{0pt}scalarly measurable multifunction $\Gamma\colon\Omega\rightrightarrows X^*$ with $\mathit{w}^*$-\hspace{0pt}compact and convex values is {\sc $w^*$-integrable} if it is $\mathit{w}^*$-scalarly integrable and for every $A\in\Sigma$ there is a $\mathit{w}^*$-\hspace{0pt}compact and convex set $C_A\subset X^*$ such that
$$
s(x,C_A)=\int_A s(x,\Gamma)d\mu\quad\text{for every $x\in X$}.
$$
Then the set $C_A$ is called the {\sc $w^*$-integral} of $\Gamma$ over $A$ and denoted by $w^*$-$\int_A\Gamma d\mu$, where the set indication is omitted when $A=\O$.
\end{Definition}

The classical separation theorem tells us that the $w^*$-integral $w^*$-$\int_A\Gamma d\mu$ is uniquely determined for every $A\in\Sigma$. Some useful facts on $w^*$-integrability of multifunctions are summarized in the next proposition taken from \cite[Theorem~2.7 and Theorem~4.5]{ckr11} and from \cite[Theorem 6.7]{mu13}.

\begin{Proposition} [\bf properties of $w^*$-integrable convex-valued multifunctions]\label{fac} In the setting of Definition~{\rm\ref{gel-convex}} the following hold:

{\bf(i)} Every $\mathit{w}^*$-\hspace{0pt}scalarly measurable multifunction with bounded values admits
a $\mathit{w}^*$-\hspace{0pt}scalarly measurable $\mathit{w}^*$-\hspace{0pt}almost selector.

{\bf(ii)} Every $\mathit{w}^*$-scalarly integrable multifunction with $\mathit{w}^*$-compact and convex values is $\mathit{w}^*$-integrable.

{\bf(iii)} A $\mathit{w}^*$-\hspace{0pt}scalarly measurable multifunction $\Gamma\colon\Omega\rightrightarrows X^*$ with $\mathit{w}^*$-\hspace{0pt}compact and convex values is $w^*$-integrable if and only if every $\mathit{w}^*$-\hspace{0pt}scalarly measurable, $\mathit{w}^*$-\hspace{0pt}almost selector of $\Gamma$ is $w^*$-integrable. In this case we have that the $w^*$-integral of $\Gamma$ is given by
\begin{eqnarray}\label{eq1}
w^*\mbox{-}\int_A\Gamma d\mu=\left\{\int_Af d\mu\mid f\in\mathcal{S}^{1,\mathit{w}^*}_\Gamma\right\}\quad\text{for every $A\in\Sigma$}.
\end{eqnarray}
\end{Proposition}

We recall some definitions and facts from measure space theory; see, e.g., \cite{fr08}. Denote by $\mathcal{N}(\mu)$ the null ideal of $\Sigma$, i.e., $\mathcal{N}(\mu):=\{N\in\Sigma\mid\mu(N)=0\}$. A measure $\mu$ on a $\sigma$-algebra $\Sigma$ is \textit{$\kappa$-\hspace{0pt}additive} if for every pairwise disjoint family $\E\subset\Sigma$ with $|\E|<\kappa$ we have $\bigcup\E\in\Sigma$ and $\mu(\bigcup\E)=\sum_{A\in\E}\mu(A)$, where the sum is understood as $\sup_{{\cal F}\subset{\cal E},\,|{\cal F}|<\infty}\sum_{A\in{\cal F}}\mu(A)$. The \textit{additivity} $\kappa(\mu)$ of $\mu$ is the largest cardinal of $\kappa$ for which $\mu$ is $\kappa$-\hspace{0pt}additive, or it is $\infty$ if $\mu$ is $\kappa$-\hspace{0pt}additive for every $\kappa$. It follows from the definition that $\kappa(\mu)\ge\aleph_1$ for every measure $\mu$, where $\aleph_1$ signifies the first uncountable cardinal. A useful representation taken from \cite[Proposition~521A]{fr08} is
\begin{eqnarray*}
\kappa(\mu)=\min\left\{|{\cal E}|\mid{\cal E}\subset\mathcal{N}(\mu),\,\bigcup{\cal E}\not\in\Sigma\right\}.
\end{eqnarray*}
Denote further by $\mathrm{dens}(X)$ the \textit{density} of the Banach space $X$, i.e., the smallest cardinal of the form $|D|$, where $D$ is a dense subset of the open unit ball $B_X$ of $X$ with respect to the norm topology.

\begin{Proposition}[\bf equivalence of $\mathit{w}^*$-integrable selectors]\label{dens}
Let $(\Omega,\Sigma,\mu)$ be a complete finite measure space, $X$ be a Banach space, and $\Gamma:\Omega\rightrightarrows X^*$ be a $\mathit{w}^*$-scalarly integrable multifunction with $\mathit{w}^*$-\hspace{0pt}compact and convex values. If $\mathrm{dens}(X)<\kappa(\mu)$, then $\mathcal{S}^1_\Gamma=\mathcal{S}^{1,\mathit{w}^*}_\Gamma$.
\end{Proposition}
\begin{proof}
By Proposition \ref{fac}, $\mathcal{S}^{1,\mathit{w}^*}_\Gamma$ is nonempty. Choose any $f\in \mathcal{S}^{1,\mathit{w}^*}_\Gamma$. Pick a dense subset $\{ x_\alpha \}_{\alpha<\mathrm{dens}(X)}$ from $B_X$ and let $E_\alpha\in \Sigma$ with $\mu(\Omega\setminus E_\alpha)=0$ be such that $\langle f(\omega),x_\alpha \rangle\le s(x_\alpha,\Gamma(\omega))$ for every $\omega\in E_\alpha$ and $\alpha$. It follows from $\kappa(\mu)$-\hspace{0pt}additivity that $E:=\bigcap_{\alpha<\mathrm{dens}(X)}E_\alpha$ belongs to $\Sigma$ and $\mu(\Omega\setminus E)=0$. Since $\Gamma$ has $\mathit{w}^*$-compact values and $\{ x_\alpha \}_{\alpha<\mathrm{dens}(X)}$ is dense in $B_X$, the above inequality yields $\langle f(\omega),x \rangle\le s(x,\Gamma(\omega))$ for every $\omega\in E$ and $x\in X$. In view of the $w^*$-compactness and the convexity of $\Gamma(\omega)$, the separation theorem guarantees that $f(\omega)\in \Gamma(\omega)$ for every $\omega\in \Omega$. Therefore, $f\in \mathcal{S}^1_\Gamma$.
\end{proof}

If $X$ is a separable Banach space, then $\mathrm{dens}(X)=\aleph_0$ and hence the density condition in Proposition \ref{dens} obviously holds due to $\kappa(\mu)\ge\aleph_1$. However, the validity of the density condition goes far beyond separability. In particular, if $\Omega$ is a set of cardinality $\kappa$ while $\mu$ is a probability measure on the power set of $\Omega$ with $\mu(\{\omega\})=0$ for every $\omega\in\Omega$, then we have that $\mu$ is $\kappa$-\hspace{0pt}additive; see, e.g., \cite[Proposition~1194]{da14}.
\vspace*{0.05in}

It follows from Definition~\ref{gel}, Proposition~\ref{fac}(ii), and the above relationships between the sets $\mathcal{S}^1_\Gamma$ and $\mathcal{S}^{1,\mathit{w}^*}_\Gamma$ that in the $\mathit{w}^*$-compact and convex-valued setting of Definition~\ref{gel-convex} we have the inclusion
\begin{eqnarray}\label{gel-rel}
\int\Gamma d\mu\subset w^*\mbox{-}\int\Gamma d\mu,
\end{eqnarray}
which holds as equality if $\mathrm{dens}(X)<\kappa(\mu)$ in view of Proposition \ref{dens}.
\vspace*{0.05in}

Now we are ready to derive the major result of this section, which establishes the $\mathit{w}^*$-\hspace{0pt}compactness and {\em convexity} of the $\mathit{w}^*$-\hspace{0pt}closure of the Gelfand integral \eqref{eq0} of multifunctions under the nonatomicity hypothesis. It improves the result of \cite[Theorem~2]{kh85} in the sense of removing the separability assumption on $X$ and weakening the graph measurability of $\Gamma$ in $\Sigma\otimes \mathrm{Borel}(X^*,\mathit{w}^*)$ to the $\mathit{w}^*$-\hspace{0pt}measurability in the Asplund space setting.

\begin{Theorem} [\bf convexity of the closure of the Gelfand integral]\label{thm1} Let $(\Omega,\Sigma,\mu)$ be a complete and nonatomic finite measure space, and let $X$ be an Asplund space with $\mathrm{dens}(X)<\kappa(\mu)$. If $\Gamma\colon\Omega\rightrightarrows X^*$ is an integrably bounded, $\mathit{w}^*$-\hspace{0pt}scalarly measurable, Gelfand integrable multifunction with $\mathit{w}^*$-\hspace{0pt}closed values, then the set $\overline{I(\Gamma)}^{\,\mathit{w}^*}$ is $\mathit{w}^*$-\hspace{0pt}compact and convex satisfying the relationship $\overline{I(\Gamma)}^{\,\mathit{w}^*}=I(\overline{\mathrm{co}}^{\,\mathit{w}^*}\Gamma)$.
\end{Theorem}
\begin{proof} Proceeding as in \cite[Claim 3 of the Proof of Theorem 2]{kh85}, we conclude that $\overline{I(\Gamma)}^{\,\mathit{w}^*}$ is $\mathit{w}^*$-\hspace{0pt}compact and convex. Employing now the relationship
$$
w^*\mbox{-}\int\overline{\mathrm{co}}^{\,\mathit{w}^*}\Gamma d\mu=I(\overline{\mathrm{co}}^{\,\mathit{w}^*}\Gamma)
$$
that is due to Proposition \ref{dens} and the $w^*$-integral representation \eqref{eq1} in Proposition~\ref{fac}, we get
\begin{eqnarray*}
\begin{array}{ll}
\disp s(x,I(\overline{\mathrm{co}}^{\,\mathit{w}^*}\Gamma))&=s\left(x,\disp w^*\mbox{-}\int\overline{\mathrm{co}}^{\,\mathit{w}^*}\Gamma d\mu\right)=
\disp\int s(x,\overline{\mathrm{co}}^{\,\mathit{w}^*}\Gamma)d\mu\\
&=\disp\int s(x,\Gamma)d\mu=\sup_{f\in\mathcal{S}^1_\Gamma}\int\langle f(\omega),x \rangle d\mu\\
&=\disp\sup_{f\in\mathcal{S}^1_\Gamma}\left\langle \int fd\mu,x \right\rangle=s(x,I(\Gamma))=s\left(x,\overline{I(\Gamma)}^{\,\mathit{w}^*}\right)
\end{array}
\end{eqnarray*}
for every $x\in X$, where the second equality follows from Definition~\ref{gel-convex} of the $w^*$-integrals and the forth equality uses Proposition~\ref{lem04} (where we employ the Asplund property). Since $\overline{I(\Gamma)}^{\,\mathit{w}^*}$ and $I(\overline{\mathrm{co}}^{\,\mathit{w}^*}\Gamma)$ are $\mathit{w}^*$-\hspace{0pt}compact and convex, we thus obtain the claimed equality $\overline{I(\Gamma)}^{\,\mathit{w}^*}=I(\overline{\mathrm{co}}^{\,\mathit{w}^*}\Gamma)$ by using the convex separation theorem.
\end{proof}

It has been well-recognized in the literature that the finite-dimensional formulation of the Lyapunov convexity theorem fails to hold in infinite-\hspace{0pt}dimensional Banach spaces without the additional {\em closure} operation; see, e.g., \cite[Examples IX.1.1 and 1.2]{du77}). Due to this, the nonatomicity of the measure $\mu$ is insufficient to guarantee that the Gelfand integral $I(\Gamma)$ is $\mathit{w}^*$-\hspace{0pt}compact and convex, and thus the closure operation is inevitable in Theorem~\ref{thm1}. To overcome this difficulty, we introduce the saturation requirement on measure spaces along the lines of \cite{fk02,hk84,ks09,ma42}, which ensures the validity of the Lyapunov convexity theorem with {\em no closure} operation in dual spaces to separable Banach spaces; see \cite{ks15a}). Furthermore, the saturation property is also necessary for the $\mathit{w}^*$-\hspace{0pt}compactness and the convexity of the Gelfand integral of multifunctions as well as for the Lyapunov convexity theorem in separable Banach spaces; see \cite{ks13}.

Recall that a finite measure space is said to be \textit{essentially countably generated} if its $\sigma$-\hspace{0pt}algebra can be generated by a countable number of subsets together with measure null sets. If it is not the case, then the measure space under consideration is said to be \textit{essentially uncountably generated}. Let $\Sigma_E=\{A\cap E\mid A\in\Sigma\}$ be the restriction of the $\sigma$-\hspace{0pt}algebra $\Sigma$ to a measurable set $E\in\Sigma$. Denote by $L^1_E(\mu)$ the space of $\mu$-integrable functions defined on the restricted measure space $(E,\Sigma_E,\mu)$.

\begin{Definition} [\bf measure saturation]\label{satur} A finite measure space $(\Omega,\Sigma,\mu)$ is {\sc saturated} if the space $L^1_E(\mu)$ is nonseparable with respect to the $L^1$-\hspace{0pt}norm topology for every $E\in\Sigma$ with $\mu(E)>0$.
\end{Definition}

This property, which surely implies nonatomicity, and the results related to it have been found as an important tool for various economic applications and probabilistic models; see, in particular, the papers on saturation listed above. Several equivalent definitions of the saturation property are known; see, e.g., \cite{ks09}. One of the remarkable characterizations of saturation is as follows: a finite measure space $(\Omega,\Sigma,\mu)$ is saturated if and only if $(E,\Sigma_E,\mu)$ is essentially uncountably generated for every $E\in\Sigma$ with $\mu(E)>0$.

\begin{Proposition}[\bf convexity of the Gelfand integral]\label{saturation} Let $(\Omega,\Sigma,\mu)$ be a complete and saturated finite measure space, and let $X$ be a separable Asplund space. If $\Gamma\colon\Omega\rightrightarrows X^*$ is an integrably bounded and $\mathit{w}^*$-\hspace{0pt}scalarly measurable multifunction with $\mathit{w}^*$-\hspace{0pt}closed values, then the set $I(\Gamma)$ is $\mathit{w}^*$-\hspace{0pt}compact and convex satisfying the relationship $I(\Gamma)=I(\overline{\mathrm{co}}^{\,\mathit{w}^*}\Gamma)$.
\end{Proposition}
\begin{proof}
It follows from \cite[Theorem 11]{gp15} that $I(\Gamma)$ is $\mathit{w}^*$-compact and convex. Since $I(\overline{\mathrm{co}}^{\,\mathit{w}^*}\Gamma)$ is $\mathit{w}^*$-compact and convex, and $I(\Gamma)\subset I(\overline{\mathrm{co}}^{\,\mathit{w}^*}\Gamma)=\overline{I(\Gamma)}^{\,\mathit{w}^*}$ by Theorem \ref{thm1}, we obtain the equality $I(\Gamma)=I(\overline{\mathrm{co}}^{\,\mathit{w}^*}\Gamma)$.
\end{proof}

\section{Subdifferentiation of Integral Functionals}

This section concerns integral functionals of type \eqref{int-f} defined via the integral of the integrand functions $\ph(x,\omega)$ that are locally Lipschitzian in $x$ and $\mu$-measurable in $\omega$. For convenience and without loss of generality in the proofs below, we suppose in what follows that $\ph(x,\cdot)$ is $\mu$-integrable for any $x$ around the reference point $\ox$. According to the discussions in Section~1, our goal is to derive appropriate subdifferential versions of the classical Leibniz rule on differentiation under the integral sign, where the integration of set-valued subdifferential mappings is understood accordingly either in the Gelfand sense of Definition~\ref{gel}, or in the modified sense of the $w^*$-integral from Definition~\ref{gel-convex} for convex-valued multifunctions.

We begin with the generalized differential constructions by Clarke \cite{cl83} for Lipschitz continuous functions on arbitrary Banach spaces. Given $\phi\colon X\to\oR$ locally Lipschitzian around $\ox$, recall first its {\em generalized directional derivative} at $\ox$ in the direction $h\in X$ defined by
\begin{eqnarray}\label{dd}
\phi^\circ(\bar{x};h):=\disp\limsup_\substack{x\to\ox}{\theta\downarrow 0}\frac{\phi(x+\theta h)-\phi(x)}{\theta}.
\end{eqnarray}
A crucial property of the function $h\mapsto\phi^\circ(\bar{x};h)$ is its automatic convexity, which is the source---together with the convex separation theorem---of nice calculus rules for it as well as for the {\em generalized gradient} (known also as the convexified or Clarke subdifferential) of $\phi$ at $\ox$ induced by \eqref{dd} via the conventional duality scheme
\begin{eqnarray}\label{gg}
\partial^\circ\phi(\bar{x}):=\big\{x^*\in X^*\big|\;\langle x^*,h\rangle\le\phi^\circ(\bar{x};h)\;\mbox{ for every }\;h\in X\big\}
\end{eqnarray}
of generating subdifferentials from directional derivatives. It is easy to observe that the set $\partial^\circ\phi(\bar{x})$ is nonempty, convex, and $\mathit{w}^*$-\hspace{0pt}compact in $X^*$. Furthermore, the convexity of $\phi^\circ(\ox;\cdot)$ easily implies by convex separation that \eqref{dd} is the support function \eqref{supp} of the generalized gradient, i.e., we have
\begin{eqnarray}\label{supp-g}
\phi^\circ(\bar{x};h)=s\big(h,\partial^\circ\phi(\bar{x})\big)\quad\text{for every $h\in X$}.
\end{eqnarray}
Recall that $\phi$ is {\em regular} at $\ox$ if the classical directional derivative
\begin{eqnarray}\label{dd0}
\phi'(\bar{x};h):=\lim_{\theta\downarrow 0}\frac{\phi(\bar{x}+\theta h)-\phi(\bar{x})}{\theta}
\end{eqnarray}
exists and agrees with \eqref{dd}, i.e., $\phi'(\bar{x};h)=\phi^\circ(\bar{x};h)$ for every $h\in X$. The class of regular functions contains smooth and convex ones as well as their reasonable extensions and compositions; see \cite{cl83} for the facts reviewed above.

Having an integrand $\ph\colon X\times\O\to\oR$ in \eqref{int-f} defined on the product of a Banach space $X$ and a finite measure space $(\O,\Sigma,\mu)$ and assuming that $\ph(\cdot,\omega)$ is locally Lipschitzian around some point $\ox$, we consider the generalized directional derivative $\ph^\circ(\ox,\omega;h)$ of $\ph$ at $\ox$ in the direction $h\in X$ for each fixed $\omega\in\O$ and then the generalized gradient $\partial^\circ\ph(\ox,\omega)$ of $\ph$ at $\ox$ induced by $\ph^\circ(\ox,\omega;h)$ in \eqref{gg}. It follows from the definition of the $\mathit{w}^*$-scalar measurability of the multifunction and construction \eqref{gg} of the generalized gradient that the multifunction $\partial^\circ\ph(\ox,\cdot)$ is $\mathit{w^*}$-scalarly measurable (resp.\ $\mathit{w}^*$-scalarly integrable) if and only if the generalized directional derivative function $\varphi^\circ(\bar{x},\cdot;h)$ is measurable (resp.\ integrable) for each $h\in X$. This surely holds when $\varphi$ is regular at $\ox$, since in this case we get $\varphi^\circ(\bar{x},\omega;h)=\varphi'(\bar{x},\omega;h)$, and thus $\varphi^\circ(\bar{x},\cdot;h)$ is measurable as the pointwise limit \eqref{dd0} of a sequence of measurable functions. Since the multifunction $\partial^\circ\ph(\ox,\cdot)$ has convex and $w^*$-compact values in $X^*$, we proceed in what follows with evaluating the generalized gradient of the integral functional \eqref{int-f} at $\ox$ in terms of the {\em $w^*$-integral} of the generalized gradient mapping $\partial^\circ\ph(\ox,\cdot)$.

The next theorem justifies a generalized Leibniz rule in this vein for the case of general Banach spaces $X$. Its Bochner integral counterpart was obtained in \cite[Theorem~2.7.2]{cl83} for separable Banach spaces, and the separability assumption seems to be indispensable in the proof therein.

\begin{Theorem}[\bf generalized gradients of integral functionals in Banach spaces]\label{cla}
Let $(\Omega,\Sigma,\mu)$ be a complete finite measure space, and let $X$ be an arbitrary Banach space. Consider the integral functional $I_\ph$ in \eqref{int-f}, where the integrand $\varphi\colon X\times\Omega\to\oR$ is Lipschitz continuous in $x$ on some neighborhood $U$ of $\ox$ for every $\omega\in\O$ with an integrable Lipschitz modulus being uniform on $U$. Assume also that $\varphi(x,\cdot)$ is measurable for every $x\in U$ and that the multifunction $\partial^\circ\varphi(\bar{x},\cdot)$ is $\mathit{w}^*$-\hspace{0pt}scalarly measurable. Then $I_\ph$ is locally Lipschitzian around $\ox$ and we have the inclusion
\begin{eqnarray}\label{cl-inc}
\partial^\circ I_\varphi(\bar{x})\subset w^*\mbox{-}\disp\int\partial^\circ\varphi(\bar{x},\omega)d\mu=\disp\int\partial^\circ\varphi(\bar{x},\omega)d\mu.
\end{eqnarray}
If furthermore $\varphi(\cdot,\omega)$ is regular at $\bar{x}$ for every $\omega\in\Omega$, then $I_\varphi$ is also regular at this point and \eqref{cl-inc} holds as equality.
\end{Theorem}
\begin{proof}
Observe first that the local Lipschitz continuity of $I_\ph$ around $\ox$ follows from the one for $\ph(\cdot,\omega)$ with the uniformly integrable Lipschitz modulus on the finite measure space. Further, the $\mathit{w}^*$-scalarly integrability of the multifunction $\partial^\circ\varphi(\bar{x},\cdot)$ implies its $\mathit{w}^*$-integrability by Proposition \ref{fac}(ii) and the $\mu$-integrability of the Lipschitz modulus readily yields the $\mu$-integrability for the function $\omega\mapsto\varphi^\circ(\bar{x},\omega;h)$. To verify now inclusion \eqref{cl-inc}, fix $\bar{x}\in X$ and $h\in X$. Take any $x_n\to\bar{x}$ and $\theta\downarrow 0$ satisfying
$$
I_\varphi^\circ(\bar{x};h)=\disp\lim_{n\to\infty}\frac{I_\varphi(x_n+\theta_nh)-I_\varphi(x_n)}{\theta_n}.
$$
By definition of the generalized directional derivative \eqref{dd} we have
\begin{eqnarray*}
\begin{array}{ll}
s\big(h,\partial^\circ I_\varphi(\bar{x})\big)=I_\varphi^\circ(\bar{x};h)
&=\disp\lim_{n\to \infty}\disp\int\frac{\varphi(x_n+\theta_n h,\omega)-\varphi(x_n,\omega)}{\theta_n}d\mu\\
&=\disp\int\lim_{n\to\infty}\frac{\varphi(x_n+\theta_nh,\omega)-\varphi(x_n,\omega)}{\theta_n}d\mu\\
&\le \disp\int\varphi^\circ(\bar{x},\omega;h)d\mu\\
&=\disp\int s\big(h,\partial^\circ\varphi(\bar{x},\omega)\big)d\mu\\&=s\Big(h,\disp w^*\mbox{-}\int\partial^\circ\varphi(\bar{x},\omega)d\mu\Big)
\end{array}
\end{eqnarray*}
where we invoke the Lebesgue dominated convergence theorem in the second lines and the last equality follows from Definition~\ref{gel-convex} of the $w^*$-integral for convex-valued multifunctions. Employing the separation theorem verifies the claimed inclusion in \eqref{cl-inc}. Since every selector from $\partial^\circ\varphi(\bar{x},\cdot)$ is a $\mathit{w}^*$-almost selector by \eqref{supp-g}, we have the equality in \eqref{cl-inc} by \eqref{eq1}. The equality and regularity statements of the theorem can be justified similarly to the corresponding part in the proof of \cite[Theorem~2.7.2]{cl83}.
\end{proof}

Let us present a direct consequence of Theorem~\ref{cla}, which gives us a precise extension of the classical Leibniz rule as {\em equality} for integral functionals on Banach spaces under an appropriate differentiability assumption on the integrand. Recall that a function $\phi\colon X\to\overline{\R}$ is \textit{strictly differentiable} at $\bar{x}$ with its strict derivative $\nabla\phi(\bar{x})\in X^*$ if $\phi(\bar{x})<\infty$ and
\[
\lim_\substack{h\to 0}{x\to\bar{x}}\frac{\phi(x+h)-\phi(x)-\langle\nabla\phi(\bar{x}),h\rangle}{\|h\|}=0.
\]
This property implies that $\phi$ is locally Lipschitzian around $\ox$ and regular at this point with $\nabla\phi(\bar{x})=\{\nabla\phi(\ox)\}$; see \cite[Propositions~2.2.4 and 2.3.6]{cl83}.

\begin{Corollary}[\bf Leibniz rule for Gelfand integral functionals with strictly differentiable integrands]\label{lieb-sd} In addition to the general setting of Theorem~{\rm\ref{cla}} suppose that $\varphi(\cdot,\omega)$ is strictly differentiable at $\bar{x}$ for every $\omega\in\Omega$ with $\|\nabla\varphi(\bar{x},\cdot)\|\in L^1(\mu)$ and that $\varphi(x,\cdot)$ is measurable for every $x$ around $\bar{x}$. Then $I_\varphi$ is locally Lipschitzian around $\ox$ and we have
\begin{eqnarray}\label{cla-sd}
\nabla I_\varphi(\bar{x})=\int\nabla\varphi(\bar{x},\omega)d\mu,
\end{eqnarray}
where the strict derivative of $\varphi$ in \eqref{cla-sd} is taken with respect of $x$.
\end{Corollary}
\begin{proof} We have from the above that in the setting of this corollary all the assumptions of Theorem~\ref{cla} are satisfied and the equality holds in \eqref{cl-inc} with $\partial^\circ\varphi(\ox,\omega)=\{\nabla\varphi(\ox,\omega)\}$ for every $\omega\in\O$. It remains to mention that the $w^*$-integral in \eqref{cl-inc} reduces to the Gelfand one for single-valued integrands.
\end{proof}

Next we discuss effective conditions for the validity of the $\mathit{w}^*$-\hspace{0pt}scalar measurability assumption on the mapping $\partial^\circ\varphi(\bar{x},\cdot)$ imposed in Theorem~\ref{cla}.

\begin{Proposition}[\bf $w^*$-scalar measurability of the generalized gradient]\label{gg-mes} Assume that $\mathrm{dens}(X)<\kappa(\mu)$ in the setting of Theorem~{\rm\ref{cla}}. Then the generalized gradient mapping $\partial^\circ\varphi(\bar{x},\cdot)$ is $\mathit{w}^*$-\hspace{0pt}scalarly measurable.
\end{Proposition}
\begin{proof} As discussed, it suffices to show that the measurability of the generalized directional derivative $\omega\mapsto\varphi^\circ(\bar{x},\omega;h)$. To proceed, pick a dense subset $\{x_\alpha\}_{\alpha<\mathrm{dens}(X)}$ of $B_X$ and a countable dense set $\{\theta_k\}$ in $\R$, and then choose a sequence of $\varepsilon_k\downarrow 0$ as $k\to\infty$. It follows from the density of $\{x_\alpha\}_{\alpha<\mathrm{dens}(X)}$ and $\{\theta_k\}$ and from the continuity of $\varphi(\cdot,\omega)$ that
\begin{eqnarray*}
\begin{array}{ll}
\varphi^\circ(\bar{x},\omega;h)& =\disp\lim_{\varepsilon\downarrow 0}\Big[\sup_{x\in\bar{x}+\varepsilon B_X,\,\theta\in (0,\varepsilon)}\disp\frac{\varphi(x+\theta h,\omega)-\varphi(x,\omega)}{\theta}\Big]\\
&=\disp\lim_{k\to\infty}\bigg[\sup_{\begin{subarray}{c}\alpha<\mathrm{dens}(X)\\n\in\N:\theta_n\in(0,\varepsilon_k) \end{subarray}}\frac{\varphi((\bar{x}+\varepsilon_k x_\alpha)+\theta_n h,\omega)-\varphi(\bar{x}+\varepsilon_k x_\alpha,\omega)}{\theta_n}\bigg],
\end{array}
\end{eqnarray*}
where the scalar function
$$
\disp\omega\longmapsto\sup_{\begin{subarray}{c}\alpha<\mathrm{dens}(X)\\n\in\N:\theta_n\in(0,\varepsilon_k) \end{subarray}}\frac{\varphi((\bar{x}+\varepsilon_k x_\alpha)+\theta_n h,\omega)-\varphi(\bar{x}+\varepsilon_k x_\alpha,\omega)}{\theta_n}
$$
is measurable for each $k\in\N$ due to $\mathrm{dens}(X)\times \aleph_0<\kappa(\mu)\times\aleph_0=\kappa(\mu)$; cf.\ \cite[Proposition~521B]{fr08}. Thus it shows that the function $\omega\mapsto\varphi^\circ(\bar{x},\omega;h)$ is measurable as the pointwise limit of a sequence of measurable functions.
\end{proof}

Next we proceed with establishing the extensions of the Leibniz rule \eqref{lem} from \cite[Lemma~6.18]{mo06} outlined in Section~1. Our standing setting in this part is the class of {\em Asplund} spaces $X$, which may be nonreflexive and nonseparable. Given a function $\phi\colon X\to\oR$ locally Lipschitzian around $\ox$, recall that the {\em limiting subdifferential} (known also as the basic, general, or Mordukhovich one) of $\phi$ at $\ox$ is defined by
\begin{eqnarray}\label{bs}
\begin{array}{ll}
\partial\phi(\ox):=\Big\{x^*\in X^*\Big|&\exists\,\mbox{ sequences }\;x_k\to\ox,\;x^*_k\st{w^*}{\to}x^*\;\mbox{ such that}\\
&\disp\liminf_{x\to x_k}\frac{\phi(x)-\phi(x_k)-\la x^*_k,x-x_k\ra}{\|x-x_k\|}\ge 0\Big\}.
\end{array}
\end{eqnarray}

In \cite{mo06a}, the reader can find a comprehensive theory for the subdifferential \eqref{bs} in Asplund (and partly in general Banach) spaces with its various applications in \cite{mo06}. Note that $\partial\phi(\ox)\ne\emp$ for every locally Lipschitzian function on an Asplund space, and in fact this property is a characterization of Asplund spaces; see \cite[Corollary~3.25]{mo06a}. We have $\partial\phi(\ox)=\{\nabla\phi(\ox)\}$ when $\phi$ is strictly differentiable at $\ox$, and \eqref{bs} reduces to the subdifferential of convex analysis for convex functions $\phi$. But it may be heavily nonconvex even for simplest nonconvex functions on $\R$ as, e.g., in the case of $\phi(x)=-|x|$ where $\partial\phi(0)=\{-1,1\}$. Furthermore, it follows from \cite[Theorem~3.57]{mo06} that
\begin{eqnarray}\label{mor1}
\partial^\circ\phi(\bar{x})=\overline{\mathrm{co}}^{\,\mathit{w}^*}\partial\phi(\bar{x})
\end{eqnarray}
whenever $\phi$ is locally Lipschitzian around $\ox$ on an Asplund space $X$. It is worth mentioning here that in spite of (in fact due to) its nonconvexity the limiting subdifferential \eqref{bs} enjoys {\em full calculus} on Asplund spaces that is based on {\em variational/extremal principles} of variational analysis; see \cite{mo06a}.

Consider now the Gelfand integral \eqref{eq0} of the nonconvex subdifferential mapping $\partial\varphi(x,\cdot)\colon\Omega\rightrightarrows X^*$ given by
$$
I\big(\partial\varphi(x,\cdot)\big)=\int\partial\varphi(x,\omega)d\mu,
$$
where the notation $\partial\varphi(x,\omega)$ stands for $\partial(\varphi(\cdot,\omega))(x)$.

The following result is an extension of \cite[Lemma~6.18]{mo06} from reflexive and separable to general Asplund spaces $X$ and from Bochner to Gelfand integration with respect to complete nonatomic measures.

\begin{Theorem}[\bf limiting subdifferential of integral functionals in Asplund spaces]\label{thm2} Let $(\Omega,\Sigma,\mu)$ be a complete and nonatomic finite measure space, and let $X$ be an Asplund space with $\mathrm{dens}(X)<\kappa(\mu)$. Let $\varphi:X\times\Omega\to\oR$ be an integrand in \eqref{int-f} such that the function $\varphi(\cdot,\omega)$ is Lipschitz continuous on a neighborhood $U$ of $\bar{x}$ for every $\omega\in\Omega$ with a uniformly on $U$ integrable Lipschitz modulus, while the function $\varphi(x,\cdot)$ is measurable for every $x\in U$ and such that the mapping $\partial\varphi(\bar{x},\cdot)$ is $\mathit{w}^*$-scalarly measurable and Gelfand integrable. Then we have the inclusion
\begin{eqnarray}\label{lim-inc}
\partial I_\varphi(\bar{x})\subset\overline{\int\partial\varphi(\bar{x},\omega)d\mu}^{\,\mathit{w}^*},
\end{eqnarray}
which holds as equality provided that $\ph(\cdot,\omega)$ is regular at $\ox$ for all $\omega\in\O$.
\end{Theorem}
\begin{proof}
As shown in the proof of Theorem~\ref{cla}, the integral functional $I_\ph$ in \eqref{int-f} is locally Lipschitzian around $\ox$ under the assumptions made. Moreover, the local Lipschitz continuity of $\ph(\cdot,\omega)$ with a uniformly integrable modulus implies the uniform boundedness of $\partial\ph(\cdot,\omega)$ on $U$ and hence the integrable boundedness of the multifunction $\partial\varphi(\bar{x},\cdot)$ with $\mathit{w}^*$-\hspace{0pt}compact values; see \cite[Corollary~1.81]{mo06a}. It follows from \eqref{mor1} that
$$
s\big(h,\partial\varphi(\bar{x},\omega)\big)=s\big(h,\overline{\mathrm{co}}^{\,\mathit{w}^*}\partial\varphi(\bar{x},\omega)\big)
=s\big(h,\partial^\circ\varphi(\bar{x},\omega)\big)
$$
for every $h\in X$ and $\omega\in\Omega$, which implies the $\mathit{w}^*$-\hspace{0pt}scalar measurability of $\partial\varphi(\bar{x},\cdot)$ and that of $\partial^\circ\varphi(\bar{x},\cdot)$ are equivalent. Further, the Gelfand integrability of $\partial\varphi(\bar{x},\cdot)$ yields that of $\partial^\circ\varphi(\bar{x},\cdot)$, which verifies the inclusion \eqref{cl-inc} by Theorem~\ref{cla}. On the other hand, we have by Theorem~\ref{thm1} under the imposed nonatomicity assumption that $I(\overline{\mathrm{co}}^{\,\mathit{w}^*}\partial\varphi(\bar{x},\cdot))=\overline{I(\partial\varphi(\bar{x},\cdot))}^{\,\mathit{w}^*}$. Combining this with \eqref{eq1} and \eqref{mor1} leads us to the relationships
\begin{eqnarray*}
\begin{array}{ll}
\partial I_\varphi(\bar{x})\subset\partial^\circ I_\varphi(\bar{x})&\subset\disp w^*\mbox{-}\int\partial^\circ\varphi(\bar{x},\omega)d\mu
=\disp w^*\mbox{-}\int\overline{\mathrm{co}}^{\,\mathit{w}^*}\partial\varphi(\bar{x},\omega)d\mu\\
&=I\big(\overline{\mathrm{co}}^{\,\mathit{w}^*}\partial\varphi(\bar{x},\cdot)\big)=\overline{I(\partial\varphi(\bar{x},\cdot))}^{\,\mathit{w}^*}
=\disp\overline{\int\partial\varphi(\bar{x},\omega)d\mu}^{\,\mathit{w}^*},
\end{array}
\end{eqnarray*}
which justifies the claimed inclusion \eqref{lim-inc}. To justify the equality therein under the regularity of $\ph(\cdot,\omega)$ at $\ox$, we can easily derive from the definitions and \eqref{mor1} that this regularity yields $\partial^\circ\ph(\ox,\omega)=\partial\ph(\ox;\omega)$ for every $\omega\in\O$. Using now the equality in \eqref{cl-inc} and the regularity of $I_\ph$ at $\ox$ established in Theorem~\ref{cla} together with the equalities obtained above in the proof of this theorem, we arrive at the relationships
\begin{eqnarray*}
\partial I_\ph(\ox)=\partial^\circ I_\ph(\ox)=\disp w^*\mbox{-}\int\partial^\circ\varphi(\bar{x},\omega)d\mu=\disp\overline{\int\partial\varphi(\bar{x},\omega)d\mu}^{\,\mathit{w}^*},
\end{eqnarray*}
which verify the equality in \eqref{lim-inc} under the regularity of $\ph(\cdot,\omega)$ at $\ox$.
\end{proof}

As mentioned above, the $\mathit{w}^*$-\hspace{0pt}scalar measurability of $\partial\varphi(\bar{x},\cdot)$ assumed in Theorem~\ref{thm2} is equivalent to the one for $\partial^\circ\varphi(\bar{x},\cdot)$, and hence it holds under the validity of the condition $\mathrm{dens}(X)<\kappa(\mu)$ in Proposition~\ref{gg-mes}; in particular, when the space $X$ is separable.\vspace*{0.05in}

The next consequence of Theorem~\ref{thm2} as well as some previous results presented above establishes a Leibniz-type rule for subdifferentiation of integral functionals without convexification.

\begin{Corollary}[\bf unconvexified Leibniz rule for subdifferentiation of integral functionals]\label{cor1} Let $(\Omega,\Sigma,\mu)$ be a complete and saturated finite measure space, let $X$ be a separable Asplund space, and let the integrand $\varphi\colon X\times\Omega\to\oR$ in \eqref{int-f} satisfy the general assumptions of Theorem~{\rm\ref{thm2}}, where the Gelfand integrability of $\partial\ph(\ox,\cdot)$ is automatic. Then we have
\begin{eqnarray}\label{inc-satur}
\partial I_\varphi(\bar{x})\subset\disp\int\partial\varphi(\bar{x},\omega)d\mu=\disp\int\partial^\circ\varphi(\bar{x},\omega)d\mu,
\end{eqnarray}
where the Gelfand integral of the subdifferential mappings is equivalent to the $w^*$-integral and to the Bochner one, and where the inclusion in \eqref{inc-satur} holds as equality provided that $\ph(\cdot,\omega)$ is regular at $\ox$ for every $\omega\in\O$.
\end{Corollary}
\begin{proof}
Both conclusions in \eqref{inc-satur} as well as the equality statement under regularity of $\ph(\cdot,\omega)$ at $\ox$ follow from the combination of Proposition~\ref{saturation} and the proof of Theorem~\ref{thm2}. As mentioned above, the Gelfand integral in \eqref{inc-satur} agrees with the $w^*$-integral due to the separability of $X$, while both of these integrals of the subdifferential mappings reduce to the Bochner integral in separable Asplund spaces by the discussions in Remark~\ref{rem1}.
\end{proof}

\begin{Remark}[\bf separable reduction in Asplund spaces]\label{sep-red}
{\rm There is the method of {\em separable reduction} in theory of Asplund spaces, which provides the possibility to elevate results involving the so-called ``Fr\'echet subdifferential" (known also as the regular or viscosity subdifferential, or the presubdifferential in variational analysis) of $\phi\colon X\to\oR$ at $\ox$ defined by
\begin{eqnarray*}
\Hat\partial\phi(\ox):=\Big\{x^*\in X^*\Big|\;\disp\liminf_{x\to\ox}\frac{\phi(x)-\phi(\ox)-\la x^*,x-\ox\ra}{\|x-\ox\|}\ge 0\Big\}
\end{eqnarray*}
from separable to general Asplund spaces. It has been used to extend a number of important results of generalized differentiation in the Fr\'echet sense established in separable Asplund spaces to arbitrary Asplund spaces without any separability-like assumptions. The reader can find more information on the usage of this method, largely developed by Mari\'an Fabian, in \cite[Section~2.2.2]{mo06a} that is based on \cite{fm02}; see also \cite{cf16} for the recent developments. Remembering definition \eqref{bs} of the limiting subdifferential, we {\em conjecture} that the results of Theorem~\ref{thm2} and Corollary~\ref{cor1} related to separable Asplund spaces $X$ can be extended to the arbitrary Asplund space setting by implementing the method of separable reduction. More precisely, we conjecture that the $\mathit{w}^*$-\hspace{0pt}scalar measurability assumption on $\partial\varphi(\bar{x},\cdot)$ imposed in Theorem~\ref{thm2} and satisfied in separable Asplund spaces, can be avoided. Proceeding in this way, we conjecture that the separability assumption on $X$ can be dismissed (or seriously weakened) in Corollary~\ref{cor1} for the Gelfand integral of the limiting subdifferential mapping. It should be mentioned to this end that, besides subdifferential separable reduction elaborations, more work is needed regarding the saturation property. Indeed, Proposition~\ref{saturation} used in the proof of Corollary~\ref{cor1} assumes separability. It has been shown quite recently \cite{ks15a} that the Lyapunov theorem holds in dual Banach spaces without any separability assumptions but with a certain higher-order strengthening of the saturation property. It is an interesting open question about the possibility to remove or relax the latter assumption in the nonseparable framework of Corollary~\ref{cor1} while combining with the method of separable reduction.}
\end{Remark}

\section{Bellman Equation for Problems of Stochastic Dynamic Programming in Banach Spaces}

As mentioned in Section~1, our intention in this paper is to give applications of the obtained results on subdifferentiation of integral functionals to some models of stochastic dynamic programming (DP) governed  by discrete-time systems on the infinite horizon and mainly related to the {\em Markov decision process} in which current random shocks affect a probability measure over random shocks next period as well as a cost function and a feasible set. Models of this type have been investigated in finite-dimensional spaces with various applications to economic dynamics; see, e.g., \cite{slp89} and the references therein together with further references to be presented below.

In this section we describe the basic stochastic DP model in Banach spaces and derive the corresponding Bellman equation for the optimal value function. The form of the obtained equation gives us the opportunity to employ in the next section the subdifferential results for the integral functionals established in Section~3 in order to conduct a local sensitivity analysis and derive necessary optimality conditions in stochastic DP.

To begin with the description of the random environment of the model, let $P$ be a \textit{stochastic kernel} (or a \textit{transition probability measure}) on a measurable space $(\Omega,\Sigma)$, i.e., $P\colon\Sigma\times\Omega\to[0,1]$ is such that $P(\cdot\mid\omega)$ is a probability measure on $\Sigma$ for every $\omega\in\Omega$ and that $P(A\mid\cdot)$ is a measurable function on $\Omega$ for every $A\in\Sigma$. Denote by $\omega^t=(\omega_1,\dots,\omega_t)$, $t=1,2,\ldots$, a finite sequence of random shocks in the product space $\Omega^t$ equipped with the product $\sigma$-\hspace{0pt}algebra $\Sigma^t$. Given an initial shock $\omega_0\in\Omega$, let $P^t(\cdot\mid \omega_0)$ be the probability measure on $\Sigma^t$ constructed from the iterated integrals of the stochastic kernel $P$ satisfying for every measurable rectangle $A_1\times\ldots\times A_t\in \Sigma^t$ with its probability defined for every $t=1,2,\ldots$ by
\begin{eqnarray*}
\begin{array}{ll}
&P^t(A_1\times\ldots\times A_t\mid\omega_0):\\
=&\disp\int_{A_1}\int_{A_{2}}\ldots\int_{A_{t}}P(d\omega_t\mid\omega_{t-1})P(d\omega_{t-1}\mid\omega_{t-2})\ldots P(d\omega_1\mid\omega_0).
\end{array}
\end{eqnarray*}

Let $X$ be an {\em action space}, which is assumed to be Banach in what follows. At each period a decision maker possesses a {\em random cost function} $u:X\times X\times\Omega\to\R$ and discounts the sum of expected costs with a constant rate $\beta\in[0,1)$. {\em Feasible policies} are determined by a {\em random constraint mapping} $\Gamma\colon X\times\Omega\rightrightarrows X$. Given an initial condition $(x,\omega)\in X\times\Omega$, we consider the {\em stochastic DP} problem on the {\em infinite horizon} described by:
\begin{eqnarray}\label{P}
\begin{array}{ll}
&\disp\inf\left\{u(x_0,x_1,\omega_0)+\disp\sum_{t=1}^\infty\disp\int\beta^tu(x_t,x_{t+1},\omega_t)P^t(d\omega^t\mid\omega_0)\right\}\\
&\qquad\mbox{s.t. }\;x_{t+1}\in\Gamma(x_t,\omega_t)\;\mbox{ for every }\;t=0,1,\ldots,\\
&\qquad(x_0,\omega_0)=(x,\omega)\in X\times\Omega.
\end{array}
\end{eqnarray}
Denoting by $v(x,\omega)$ the infimum value in problem \eqref{P} and ranging the initial condition $(x,\omega)$ in $v(x,\omega)$ over the whole domain $X\times\Omega$ gives us the (random) \textit{optimal value function} $v\colon X\times\Omega\to\overline{\R}$. The assumptions made below ensure that the defined value function $v(x,\omega)$ is actually real-\hspace{0pt}valued.

In accordance with the choice of topologies on $X$, two Borel $\sigma$-\hspace{0pt}algebras arise in $X$: $\mathrm{Borel}(X,\|\cdot\|)$ is the Borel $\sigma$-\hspace{0pt}algebra with respect to the norm topology and $\mathrm{Borel}(X,\mathit{w})$ is the Borel $\sigma$-\hspace{0pt}algebra with respect to the weak topology. As well known (see, e.g., \cite[p.\,67]{th75}), in the case of separable Banach spaces $X$ considered below these two notions coincide. Having this in mind, we say that the multifunction $\Gamma:X\times\Omega\rightrightarrows X$ is \textit{measurable} if for every open set $U\subset X$ in the norm topology the {\em inverse image} (or preimage)
$$
\Gamma^{-1}(U):=\{(x,\omega)\in X\times\Omega\mid\Gamma(x,\omega)\cap U\ne\emp\}
$$
of $U$ under $\Gamma$ belongs to the $\sigma$-algebra $\mathrm{Borel}(X,\|\cdot\|)\otimes\Sigma$.\vspace*{0.05in}

For our subsequent study we impose the following {\em standing assumptions} on the initial data of the stochastic DP problem \eqref{P}:\vspace*{0.05in}

{\bf (H1)} The cost function  $u:X\times X\times\Omega\to\R$ is bounded and such that $u(\cdot,\cdot,\omega)$ is continuous (resp.\ weakly continuous) on $X\times X$ for every $\omega\in\Omega$ while $u(x,y,\cdot)$ is measurable on $\Omega$ for every $(x,y)\in X\times X$.

{\bf (H2)} The constraint mapping $\Gamma:X\times\Omega\rightrightarrows X$ is measurable with compact (resp.\ weakly compact) values and such that $\Gamma(\cdot,\omega):X\rightrightarrows X$ is continuous (resp.\ weakly continuous) for every $\omega\in\Omega$ while $\Gamma(x,\cdot):\Omega\rightrightarrows X$ is measurable for every $x\in X$.\vspace*{0.05in}

Denote by $\C_b(X\times \Omega)$ (resp.\ $\C_b(X_w\times \Omega)$) the space of bounded functions $\phi\colon X\times\Omega\to\R$ endowed with the supremum norm
$$
\|\phi\|:=\disp\sup_{(x,\omega)\in X\times\Omega}|\phi(x,\omega)|
$$
such that $x\mapsto\phi(x,\omega)$ is continuous (resp.\ weakly continuous) for every $\omega\in\Omega$ while $\omega\mapsto\phi(x,\omega)$ is measurable for every $x\in X$. Since weak continuity implies norm continuity, we have that $\C_b(X_w\times\Omega)\subset\C_b(X\times\Omega)$ and that $\phi\in\C_b(X_w\times\Omega)$ is a Carath\'eodory function with respect to the norm topology of $X$. Hence it is $\mathrm{Borel}(X,\|\cdot\|)\otimes\Sigma$-\hspace{0pt}measurable whenever $X$ is separable; see \cite[Lemma~4.51]{ab06}. This implies that $u(x,y,\omega)$ is jointly measurable under the assumptions in (H1), (H2) and the separability of $X$.\vspace*{0.05in}

The next major result shows that the (optimal) value function $v(x,\omega)$ is a {\em unique} Carath\'eodory-type solution to the integral {\em Bellman equation} for the stochastic DP problem \eqref{P} involving the pointwise {\em minimization} over the constraint set $\Gamma(x,\omega)$. Note that the strong (resp.\ weak) continuity of the solution in the following theorem is in accordance with the strong (resp.\ weak) continuity assumptions in (H1) and (H2). Observe also that the form of the Bellman equation below corresponds to the discrete-time system in \eqref{P}, while its continuous-time counterpart relates to the so-called ``viscosity solutions" of the {\em Hamilton-Jacobi} (or Hamilton-Jacobi-Bellman) partial differential equation; see, e.g., \cite{fs93} with the references therein.

\begin{Theorem} [\bf integral Bellman equation in stochastic DP]\label{bell-eq} Let $X$ be a separable Banach space, and let the assumptions in {\rm(H1)} and {\rm(H2)} hold. Then there exists a unique function $v$ in $\C_b(X\times\Omega)$ {\rm (}resp.\ in $\C_b(X_w\times\Omega)${\rm )} satisfying the following Bellman equation for every $(x,\omega)\in X\times\Omega$:
\begin{equation}\label{bell}
v(x,\omega)=\min_{y\in\Gamma(x,\omega)}\left\{u(x,y,\omega)+\beta\int v(y,\omega')P(d\omega'\mid\omega)\right\}.
\end{equation}
\end{Theorem}
\begin{proof} Define the {\em Bellman operator} $T$ on $\C_b(X\times\Omega)$ (resp.\ $\C_b(X_w\times\Omega)$) by
\begin{equation}\label{eq2}
(T\phi)(x,\omega):=\min_{y\in\Gamma(x,\omega)}\left\{u(x,y,\omega)+\beta\int\phi(y,\omega')P(d\omega'\mid\omega)\right\}.
\end{equation}
To demonstrate that $T\phi$ belongs to $\C_b(X\times\Omega)$ (resp.\ to $\C_b(X_w\times\Omega)$) for every $\phi$ in $\C_b(X\times\Omega)$ (resp.\ $\C_b(X_w\times\Omega)$), we apply the measurable maximum theorem from \cite[Theorem~18.19]{ab06} to the objective function on the right-\hspace{0pt}hand side in \eqref{eq2}, which is measurable in $(x,y,\omega)$ and continuous (resp.\ weakly continuous) in $(x,y)$, and to the measurable multifunction $\Gamma$ in order to justify the measurability of $T\phi$. Picking then any $\omega\in\Omega$, apply the Berge maximum theorem from \cite[Theorem~17.31]{ab06} to the objective function in \eqref{eq2}, which is continuous (resp.\ weakly continuous) in $(x,y)$, and to the continuous (resp.\ weakly continuous) multifunction $\Gamma(\cdot,\omega)$ with compact (resp.\ weakly compact) values. This allows us to assert that $T\phi$ is continuous (resp.\ weakly continuous) in $x$. Then the Blackwell theorem from \cite[Theorem~3.53]{ab06} tells us that the Bellman operator in \eqref{eq2} is a contraction transformation of $\C_b(X\times\Omega)$ (resp.\ $\C_b(X_w\times\Omega)$) into itself. Hence for every $(x,\omega)\in X\times\O$ this operator admits and a unique fixed point $v=Tv$ solving thus the Bellman equation \eqref{bell}. The verification that $v$ is the optimal value function for the stochastic DP problem \eqref{P} is standard; see \cite[Theorem 9.2]{slp89}.
\end{proof}

Define further the \textit{policy multifunction} $G:X\times\Omega\rightrightarrows X$ by
\begin{eqnarray}\label{policy}
G(x,\omega)=\disp\Big\{y\in\Gamma(x,\omega)\Big|v(x,\omega)=u(x,y,\omega)+\beta\int v(y,\omega')P(d\omega'\mid\omega)\Big\}.
\end{eqnarray}
Employing again the Berge maximum theorem ensures that $G$ is a multifunction with compact (resp.\ weakly compact) values such that $G(\cdot,\omega)$ is upper semicontinuous (resp.\ weakly upper semicontinuous) for every $\omega\in\Omega$. Furthermore, the measurable selection result from \cite[Theorem~18.19]{ab06} tells us that for every $x\in X$ there is a measurable function $g(x,\cdot):\Omega\to X$ such that $g(x,\omega)\in G(x,\omega)$ whenever $\omega\in\Omega$  for the multifunction $G(x,\cdot):\Omega\rightrightarrows X$. Having this in mind, we say that $g:X\times\Omega\to X$ is a \textit{policy function} for \eqref{bell} if $g(x,\cdot)$ is a measurable selector of $G(x,\cdot)$ for every $x\in X$. Given a sample path $\{\omega_t\}_{t=0}^\infty$ in $\Omega$, the policy function $g$ generates an optimal random sequence $\{x_t\}_{t=0}^\infty$ in $X$ constructed recursively by $x_{t+1}:=g(x_t,\omega_t)$ for each $t=0,1,\ldots$ in the original stochastic DP problem \eqref{P}.

\section{Subdifferentiation of the Optimal Value Function in Stochastic Dynamic Programming}

The right-hand side of the Bellman equation \eqref{bell} can be treated as the optimal value function in the deterministic problem of parametric optimization
\begin{eqnarray}\label{mf}
\vt(z):=\disp\inf\big\{\phi(z,p)\big|\;p\in M(z)\big\}.
\end{eqnarray}
Functions of type \eqref{mf}, known also as {\em marginal functions}, have been investigated in variational analysis from the viewpoint of generalized differentiation with numerous applications to optimization, stability, control, and related topics; see, e.g., \cite{cl83,mo06,mn05,mnp12,mny09,t91} and the references therein. In what follows we obtain new results in this direction with taking into account the special structure of \eqref{bell} and implementing the subdifferential formulas for the integral functionals established in Section~3.

We start by establishing subdifferential results for the value function $v(x,\omega)$ in terms of the generalized gradient \eqref{gg} in Banach spaces for which the following local viability conditions on the policy multifunction \eqref{policy}.

\begin{Definition} [\bf local viability]\label{viab} Let $G\colon X\times\O\rightrightarrows X$ be the policy multifunction with $G(\ox,\omega)\ne\emp$ for some $\ox\in X$ and all $\omega\in\O$. We say that:

{\bf (i)} $G$ is {\sc locally lower viable} around $\ox$ if for every $\omega\in\Omega$ there is a neighborhood $U$ of $\ox$ such that $G(x,\omega)\cap\Gamma(x',\omega)\ne\emp$ for every $x,x'\in U$.

{\bf (ii)} $G$  is {\sc locally upper viable} around $\ox$ if for every $\omega\in\Omega$ there is a neighborhood $U$ of $\ox$ such that $G(x,\omega)\subset\Gamma(x',\omega)$ for every $x,x'\in U$.
\end{Definition}

Note that the local lower viability condition in Definition~\ref{viab}(i) is a reasonable extension of the standard interiority condition and allows us to obtain the local Lipschitz continuity of the value function $v(\cdot,\omega)$. The local upper viability property is used below for evaluating the generalized gradient of $v(\cdot,\omega)$ and deriving necessary optimality conditions in problem $(P)$ in its terms; see Theorem~\ref{env}. As shown in Theorem~\ref{mor2}, this assumption is not needed if we employ the limiting subdifferential in Asplund spaces. Observe that the local upper viability holds automatically if $\Gamma$ is independent of $x$. It also holds under the standard interiority condition; see \cite[Lemma~3.1]{alv98}.

It is worth mentioning that the results below require the separability of the space in question, since the Bellman equation in Theorem~\ref{bell-eq} is derived under this assumption. However, the arguments employed in the proofs of the subsequent results do not use separability.\vspace*{0.03in}

To formulate the next theorem, for any given $\omega\in\O$ denote by
$$
\gph G(\cdot,\omega):=\big\{(x,y)\in X\times X\big|\;y\in G(x,\omega)\big\}
$$
the graph of the policy multifunction $G(\cdot,\omega):X\rightrightarrows X$ and by $\partial_x^\circ u(\bar{x},y,\omega)$ the partial generalized gradient \eqref{gg} of $u(\cdot,y,\omega)$ at $\bar{x}\in X$ when $(y,\omega)$ is fixed. The notation $\partial_y^\circ u(x,\bar{y},\omega)$ is similar.

\begin{Theorem} [\bf generalized gradient of the value function in stochastic DP]\label{env} Let $X$ be a separable Banach space, and let $\ox\in X$ be such that $G(\ox,\omega)\ne\emp$ whenever $\omega\in\O$. Assume that the cost function $u(\cdot,\cdot,\omega)$ is locally Lipschitzian around $(\ox,y)$ for every $y\in G(x,\omega)$ with $x$ near $\ox$ and $\omega\in\Omega$ and that the policy multifunction $G$ is locally lower viable around $\ox$. Then the value function $v(\cdot,\omega)$ is locally Lipschitzian around $\ox$ for every $\omega\in\Omega$. If furthermore $G$ is locally upper viable around $\ox$ and if $u(\cdot,\cdot,\omega)$ is regular at $(\bar{x},\bar{y})\in \mathrm{gph}\,G(\cdot,\omega)$ for some $\oy\in\Gamma(\ox,\omega)$, then we have
\begin{eqnarray}\label{env1}
\partial^\circ v(\bar{x},\omega)\subset\partial^\circ_x u(\bar{x},\bar{y},\omega),\quad\omega\in\O.
\end{eqnarray}
\end{Theorem}
\begin{proof} Fix $(\bar{x},\omega)\in X\times\Omega$ and by the local lower viability of $G$ find a neighborhood $U_{\bar{x}}$ of $\bar{x}$ such that $G(x,\omega)\cap\Gamma(x',\omega)\ne\emp$ for every $x,x'\in U_{\bar{x}}$. Picking $y\in G(x,\omega)\cap\Gamma(x',\omega)$ for arbitrary points $x,x'\in U_{\bar{x}}$ gives us
$$
v(x,\omega)=u(x,y,\omega)+\beta\int v(y,\omega')P(d\omega'\mid\omega).
$$
Since $v(x',\omega)\le u(x',y,\omega)+\beta\int v(y,\omega')P(d\omega'\mid\omega)$ by \eqref{bell} and since $u(\,\cdot,\cdot,\omega)$ is locally Lipschitzian with modulus $\ell(\omega)$, we have
$$
v(x',\omega)-v(x,\omega)\le u(x',y,\omega)-u(x,y,\omega)\le\ell(\omega)\| x'-x \|.
$$
Interchanging the role of $x$ and $x'$ above tells us that
$$
|v(x,\omega)-v(x',\omega)|\le\ell(\omega)\|x-x'\|\;\mbox{ whenever }\;x,x'\in U_{\bar{x}},
$$
and hence the value function $v(\,\cdot,\omega)$ is locally Lipschitzian for every $\omega\in\Omega$.

Next we justify the inclusion \eqref{env1} under the additional assumptions made. It follows from the local upper viability of $G$ that for every $x\in U_{\bar{x}}$ and any given direction $h\in X$ we have $G(x,\omega)\subset\Gamma(x+\theta h,\omega)$ when $\th>0$ is sufficiently small. Without loss of generality choose $y\in G(x,\omega)\subset\Gamma(x+\theta h,\omega)$ for every $\theta>0$ and thus get
$$
v(x,\omega)=u(x,y,\omega)+\beta\int v(y,\omega')P(d\omega'\mid\omega).
$$
By the principle of optimality in dynamic programming we have
$$
v(x+\theta h,\omega)\le u(x+\theta h,y,\omega)+\beta\int v(y,\omega')P(d\omega'\mid\omega),
$$
\begin{equation}\label{eq5}
\frac{v(x+\theta h,\omega)-v(x,\omega)}{\theta}\le\frac{u(x+\theta h,y,\omega)-u(x,y,\omega)}{\theta}.
\end{equation}
Passing to the limit in \eqref{eq5} gives us
\begin{eqnarray*}
\begin{array}{ll}
&\disp\limsup_\substack{(x,y)\stackrel{\mathrm{gph}\,G(\cdot,\omega)}{\longrightarrow}(\bar{x},\bar{y})}{\theta\downarrow 0}\frac{u(x+\theta h,y,\omega)-u(x,y,\omega)}{\theta}\\
&\le\disp\limsup_\substack{(x,y)\to(\bar{x},\bar{y})}{\theta\downarrow 0}\frac{u(x+\theta h,y,\omega)-u(x,y,\omega)}{\theta},
\end{array}
\end{eqnarray*}
where the notation $(x,y)\stackrel{\mathrm{gph}\,G(\cdot,\omega)}{\longrightarrow}(\bar{x},\bar{y})$ means that the limit $(x,y)\to(\bar{x},\bar{y})$ in \eqref{eq5} is taken along $(x,y)\in\mathrm{gph}\,G(\cdot,\omega)$. This shows by \eqref{dd} that
\begin{eqnarray*}
\begin{array}{ll}
v^\circ(\bar{x},\omega;h)&\le u^\circ\big(\bar{x},\bar{y},\omega;(h,0)\big)=u'\big(\bar{x},\bar{y},\omega;(h,0)\big)\\\\
&=u'_x(\bar{x},\bar{y},\omega;h)=u^\circ_x(\bar{x},\bar{y},\omega;h),
\end{array}
\end{eqnarray*}
where $u'(\bar{x},\bar{y},\omega;(h,0))$ (resp.\ $u^\circ(\bar{x},\bar{y},\omega;(h,0))$) denotes the (resp.\ generalized) directional derivative of $u(\cdot,\cdot,\omega)$ at $(\bar{x},\bar{y})$ in the direction $(h,0)\in X\times X$, and where $u'_x(\bar{x},\bar{y},\omega;h)$ (resp.\ $u^\circ_x(\bar{x},\bar{y},\omega;h)$) denotes the partial (resp.\ generalized) directional derivative of $u(\cdot,\bar{y},\omega)$ at $\bar{x}$ in the direction $h\in X$. The obtained inequality is equivalent to
$$
s\big(h,\partial^\circ v(\bar{x},\omega)\big)\le s\big(h,\partial^\circ_x u(\bar{x},\bar{y},\omega)\big)\;\mbox{ for every }\;h\in X
$$
in terms of the support function \eqref{supp}. Employing finally the separation theorem, we arrive at \eqref{env1} and complete the proof.
\end{proof}

As a simple consequence of Theorem~\ref{env}, we get the following result of the strict differentiability of the value function $v(\cdot,\omega)$. It is a significant improvement upon the known results in this direction with applications to optimal economic growth models (see, e.g., \cite{acr09,am96,bs79,blv96,ki93,slp89}), since we manage to remove the convexity and the submodularity assumptions and mitigate the interior condition in the Banach space setting.

\begin{Corollary}[\bf strict differentiability of the value function in stochastic DP]\label{dif-val} Assume in the setting of Theorem~{\rm\ref{env}} that $G$ is locally upper viable around $\ox$ and that $u(\cdot,\bar{y},\omega)$ is strictly differentiable at $\bar{x}\in X$ with $(\bar{x},\bar{y})\in\mathrm{gph}\,G(\cdot,\omega)$. Then the value function $v(\cdot,\omega)$ is strictly differentiable at $\bar{x}$ and its strict derivative at $\ox$ is calculated by
$$
\nabla v(\bar{x},\omega)=\nabla_xu(\bar{x},\bar{y},\omega),\quad\omega\in\O.
$$
\end{Corollary}
\begin{proof} It immediately follows from the facts \cite{cl83} that any function strictly differentiable at a given point is regular at this point and its generalized gradient reduces to the strict derivative therein.
\end{proof}

The next theorem provides necessary optimality conditions in the stochastic DP problem \eqref{P} formulated in the form of the {\em stochastic Euler inclusion} involving the $w^*$-integral from Definition~\ref{gel-convex} and the construction of the {\em generalized normal cone} \cite{cl83} to a subset $C$ of a Banach space defined by
\begin{eqnarray}\label{cnc}
N^\circ(\oz;C):=\overline{\disp\bigcup_{\al>0}\al\,\partial^\circ{\rm dist}(\oz;C)}^{w^*}\;\mbox{ at }\;\oz\in C
\end{eqnarray}
via the generalized gradient of the Lipschitz continuous {\em distance function} to a set given by ${\rm dist}(z;C):=\inf\{\|c-z\|\;c\in C\}$. The set $C$ is {\em regular} at $\oz\in C$ if the associated distance function is regular at this point.

Note also that the theorem below imposes the integrability of the Lipschitzian modulus of $u(\cdot,\cdot,\omega)$ from Theorem~\ref{env} with respect to the measure $P(\cdot\mid\omega)$, which precisely means that $\int\ell(\omega')P(d\omega'\mid\omega)<\infty$ for such a modulus $\ell(\omega)$ whenever $\omega\in\Omega$, $y\in G(x,\omega)$, and $x$ near $\ox$.

\begin{Theorem}[\bf stochastic Euler inclusion in Banach spaces]\label{cor2}Let $X$ be a separable Banach space. In addition to the assumptions of Theorem~{\rm\ref{env}}, suppose that the Lipschitz modulus $\ell(\omega)$ of $u(\cdot,\cdot,\omega)$ is integrable with respect to a complete probability measure $P(\cdot\mid\omega)$, that $u(\cdot,\cdot,\omega)$ is regular at $(\ox,\oy)$, and that $G$ is locally upper viable around $\ox$. Then for every policy function $g:X\times\Omega\to X$ we have the following stochastic Euler inclusion:
\begin{eqnarray}\label{sei}
\begin{array}{ll}
0\in\partial^\circ_y u(\bar{x},\bar{y},\omega)&+\beta\,w^*\mbox{-}\disp\int\partial^\circ_x u\big(\bar{y},g(\bar{y},\omega'),\omega'\big)P(d\omega'\mid\omega)\\
&+N^\circ\big(\bar{y};\Gamma(\bar{x},\omega)\big),\quad\omega\in\O.
\end{array}
\end{eqnarray}
\end{Theorem}
\begin{proof} Since $\oy$ is a local optimal solution to the constrained minimization problem in \eqref{bell}, we have from \cite[Corollary to Proposition~2.4.3]{cl83} that
$$
0\in\partial^\circ_y\left(u(\cdot,\cdot,\omega)+\disp\beta\int v(\cdot,\omega')P(d\omega'\mid\omega)\right)(\ox,\oy)+N^\circ\big(\bar{y};\Gamma(\bar{x},\omega)\big),
$$
which implies by the calculus rules from \cite[Proposition~2.3.1 and Proposition~2.3.3]{cl83} the validity of the inclusion
\begin{eqnarray*}
&\partial^\circ_y\left(u(\cdot,\cdot,\omega)+\beta\disp\int v(\cdot,\omega')P(d\omega'\mid\omega)\right)(\ox,\oy)\\
&\subset\partial^\circ_y u(\bar{x},\bar{y},\omega)+\disp\beta\,\partial^\circ\Big(\int v(\cdot,\omega')P(d\omega'\mid\omega)\Big)(\oy).
\end{eqnarray*}
On the other hand, the generalized Leibniz rule from Theorem~\ref{cla} yields
$$
\partial^\circ\Big(\int v(\cdot,\omega')P(d\omega'\mid\omega\Big)(\oy)\subset\disp w^*\mbox{-}\int\partial^\circ v(\bar{y},\omega')P(d\omega'\mid\omega).
$$
Taking now any policy function $g$ and using Theorem~\ref{env} tell us that
\begin{equation}\label{eq3}
\partial^\circ v(\bar{y},\omega')\subset\partial^\circ_x u\big(\bar{y},g(\bar{y},\omega'),\omega'\big)\;\mbox{ for every }\;\omega'\in\Omega.
\end{equation}
Integrating this inclusion with respect to $P(\cdot\mid\omega)$ and combining it with the inclusions above, we arrive at the stochastic Euler inclusion \eqref{sei}.
\end{proof}

We now proceed with deriving necessary optimality conditions for the stochastic DP problem \eqref{P} by using the limiting subdifferential constructions in Asplund spaces. In this way we obtain a more refined inclusion for optimal solutions in comparison with the stochastic Euler one \eqref{sei} under somewhat different assumptions. We need the following additional definitions while referring the reader to \cite{mo06a} for more details and discussions.

Given $\omega\in\Omega$, we say that $\Gamma(\cdot,\omega):X\rightrightarrows X$ is \textit{Lipschitz-\hspace{0pt}like} around $(\bar{x},\bar{y})\in X\times X$ (or has Aubin's pseudo-Lipschitz property) if there are neighborhoods $U$ of $\bar{x}$ and $V$ of $\bar{y}$ as well as modulus $\ell(\omega)\ge 0$ such that
$$
\Gamma(x,\omega)\cap V\subset\Gamma(x',\omega)+\ell(\omega)\|x-x'\|B_X\;\mbox{ for every }\;x,x'\in U.
$$
A policy multifunction $G(\cdot,\omega):X\rightrightarrows X$ is \textit{inner semicontinuous} at $(\bar{x},\bar{y})\in \mathrm{gph}\,G(\cdot,\omega)$ if for every sequence $x_k\to\bar{x}$ there is a sequence of $y_k\in G(x_k,\omega)$ that contains a subsequence converging to $\bar{y}$, where the convergence is understood in the norm topology of $X$.

Similarly to but a bit differently from \eqref{cnc}, define the {\em basic/limiting normal cone} \cite{mo06a} to a subset $C$ of a Banach space $Z$ by
\begin{eqnarray}\label{lnc}
N(\oz;C):=\disp\bigcup_{\al>0}\al\,\partial{\rm dist}(\oz;C)\;\mbox{ at }\;\oz\in C
\end{eqnarray}
via the limiting subdifferential \eqref{bs} of the distance function. We have $N^\circ(\oz;C)=\overline{{\rm co}}^{w^*}N(\oz;C)$
provided that the space $Z$ is Asplund and that $C$ is locally closed around $\oz$; see \cite[Theorem~3.57]{mo06a}.\vspace*{0.03in}

Our first result in this direction provides a Gelfand integral condition for the limiting subdifferential of the random value function $v(\cdot,\omega)$ on the Asplund action space $X$ that generally holds without imposing either the regularity assumption on the cost function $u(\cdot,\cdot,\omega)$ or the local upper viability assumption on the policy multifunction $G$.

\begin{Theorem}[\bf Gelfand integral relations for the limiting subdifferential of the random value function in stochastic DP]\label{mor2}
Let $X$ be a separable Asplund space, and let $u(\cdot,\cdot,\omega)$ be locally Lipschitzian around $(\bar{x},\bar{y})\in\mathrm{gph}\,G(\cdot,\omega)$ for every $\omega\in\Omega$ so that its Lipschitz modulus $\ell(\omega)$ is integrable with respect to a complete and saturated probability measure $P(\cdot\mid\omega)$. Assume also that the constraint mapping $\Gamma(\cdot,\omega)$ is of closed graph and Lipschitz-\hspace{0pt}like around $(\bar{x},\bar{y})$ and that the policy multifunction $G(\cdot,\omega)$ is locally lower viable around $\ox$ and inner semicontinuous at $(\bar{x},\bar{y})$. Then there exists $x^*\in X^*$ such that we have the inclusion
\begin{eqnarray}\label{mor2-v}
\begin{array}{ll}
(x^*,0)\in\partial u(\bar{x},\bar{y},\omega)&+\left(\{0\}\times\beta\disp\int\partial v(\bar{y},\omega')P(d\omega'\mid\omega)\right)\\
&+N\big((\bar{x},\bar{y});\mathrm{gph}\,\Gamma(\cdot,\omega)\big),\quad\omega\in\O.
\end{array}
\end{eqnarray}
If furthermore the function $u(\cdot,\cdot,\omega)$ is regular at $(\ox,\oy)$ and the graph of the mapping $\Gamma_\omega(x):=\Gamma(x,\omega)$, $x\in X$, is also regular at this point for every $\omega\in\O$, then \eqref{mor2-v} is equivalent to the inclusions
\begin{eqnarray}\label{refined}
\left\{\begin{array}{ll}
0\in\partial_y u(\bar{x},\bar{y},\omega)+\beta\,\disp\int\partial v(\bar{y},\omega')P(d\omega'\mid\omega)+N\big(\bar{y};\Gamma(\bar{x},\omega)\big),\\
x^*\in\partial_x u(\ox,\oy,\omega)+N\big(\bar{x};\Gamma_\omega^{-1}(\bar{y})\big),\quad\omega\in\O.
\end{array}\right.
\end{eqnarray}
\end{Theorem}
\begin{proof} It is shown in Theorem~\ref{env} that the value function $v(\cdot,\omega)$ and the objective function $u(\cdot,\cdot,\omega)+\beta\int v(\cdot,\omega')P(d\omega'\mid\omega)$ in problem \eqref{P} are locally Lipschitzian around $\ox$ under the general assumptions of the theorem. Thus all the assumptions of \cite[Corollary~3]{mny09} obtained for the parametric optimization framework \eqref{mf} are satisfied in our setting \eqref{bell}, which allows us to find a dual element $x^*\in X^*$ such that
$$
(x^*,0)\in\partial\left(u\big(\cdot,\cdot,\omega\big)+\disp\beta\int v(\cdot,\omega')P(d\omega'\mid\omega)\right)(\ox,\oy)+N\big((\bar{x},\bar{y});\mathrm{gph}\,\Gamma(\cdot,\omega)\big).
$$
It follows from the subdifferential sum rule of \cite[Theorem~2.33]{mo06} that
\begin{eqnarray*}
&\partial\left(u(\cdot,\cdot,\omega)+\disp\beta\int v(\cdot,\omega')P(d\omega'\mid\omega)\right)(\ox,\oy)\\
&\subset\partial u(\bar{x},\bar{y},\omega)+\left[\{0\}\times\disp\beta\,\partial\Big(\int v(\cdot,\omega')P(d\omega'\mid\omega\Big)(\oy)\right].
\end{eqnarray*}
Employing further the Leibniz rule obtained in Corollary~\ref{cor1} yields
$$
\partial\Big(\int v(\cdot,\omega')P(d\omega'\mid\omega)\Big)(\oy)\subset\int\partial v(\bar{y},\omega')P(d\omega'\mid\omega)
$$
and thus verifies \eqref{mor2-v} under the general assumptions made.

Invoking now the additional regularity assumptions gives us
$$
\partial u(\bar{x},\bar{y},\omega)=\partial_x u(\bar{x},\bar{y},\omega)\times\partial_y u(\bar{x},\bar{y},\omega)
$$
by \cite[Corollary~3.44]{mo06a} and its corresponding counterpart for the limiting normal cone \eqref{lnc} that follows from the proof of \cite[Corollary~3.17]{mo06a}. Hence we arrive at \eqref{refined} and thus complete the proof of the theorem.
\end{proof}

Now we can derive from Theorem~\ref{mor2} and Theorem~\ref{env} the following necessary optimality conditions for problem \eqref{P} entirely via its initial data, where the {\em refined} stochastic Euler inclusion significantly improves the one from Theorem~\ref{cor2} under somewhat different assumptions.

\begin{Corollary}[\bf enhanced stochastic Euler inclusion in Asplund spaces]\label{enhance} In addition to the general assumptions of Theorem~{\rm\ref{mor2}} we suppose that the cost function $u(\cdot,\cdot,\omega)$ is regular at $(\bar{x},\bar{y})\in\mathrm{gph}\,G(\cdot,\omega)$ and that the policy multifunction $G$ is locally upper viable around $\ox$. Then there exists $x^*\in X^*$ such that for every policy function $g:X\times\Omega\to X$ the following enhanced stochastic Euler inclusion is satisfied:
\begin{eqnarray}\label{euler1}
\begin{array}{ll}
(x^*,0)\in\partial u\big(\bar{x},\bar{y},\omega\big)&+\left(0,\beta\disp\int\partial_x u\big(\bar{y},g(\bar{y},\omega'),\omega'\big)P(d\omega'\mid\omega)\right)\\
&+N\big((\bar{x},\bar{y});\mathrm{gph}\,\Gamma(\cdot,\omega)\big),\quad\omega\in\O.
\end{array}
\end{eqnarray}
If furthermore the set $\gph\Gamma_\omega$ is regular at $(\ox,\oy)$ for every  $\omega\in\O$, then \eqref{euler1} is equivalent to the two inclusions:
\begin{eqnarray*}
\left\{\begin{array}{ll}
0\in\partial_y u(\bar{x},\bar{y},\omega)+\beta\,\disp\int\partial_x u\big(\bar{y},g(\bar{y},\omega'),\omega'\big) P(d\omega'\mid\omega)+N\big(\bar{y};\Gamma(\bar{x},\omega)\big),\\
x^*\in\partial_x u(\ox,\oy,\omega)+N\big(\bar{x};\Gamma^{-1}_\omega(\bar{y})\big),\quad\omega\in\O.
\end{array}\right.
\end{eqnarray*}
\end{Corollary}
\begin{proof}
Integrating inclusion \eqref{eq3} with respect to the saturated probability measure $P(\cdot\mid\omega)$ and then employing Corollary~\ref{cor1} yield the relationships
\begin{eqnarray*}
\begin{array}{ll}
&\disp\int\partial v(\bar{y},\omega')P(d\omega'\mid\omega)=\int\partial^\circ v(\bar{y},\omega')P(d\omega'\mid\omega)\\
&\subset\disp\int\partial^\circ_x u\big(\bar{y},g(\bar{y},\omega'),\omega'\big)P(d\omega'\mid\omega)=\disp\int\partial_x u\big(\bar{y},g(\bar{y},\omega'),\omega'\big)P(d\omega'\mid\omega).
\end{array}
\end{eqnarray*}
Substituting the obtained inclusion into \eqref{mor2-v} gives us \eqref{euler1}. The equivalence of the latter to the last relationships in the corollary under the additional regularity assumption follows from the proof of Theorem~\ref{mor2}.
\end{proof}

Similarly to Remark~\ref{sep-red}, we conjecture that developing the method of separable reduction would allow us to {\em overcome the separability} assumption on the action space $X$ in Theorem~\ref{mor2} and its consequences involving the limiting subdifferential constructions in Asplund spaces.\vspace*{0.03in}

The results obtained above in this section address the case of general random constraint mappings $\Gamma:X\times\Omega\rightrightarrows X$. In the vast majority of applications the mapping $\Gamma$ is not arbitrary while is given in some structural form. Thus in structural cases it had become crucial to evaluate the normal cone in question via the initial data of structural constraint mappings, which largely relates to the existence of {\em strong/full calculus} for generalized normals and subgradients. Such calculus is available for the limiting constructions in Asplund spaces; see \cite{mo06a} and the references therein. In particular, this full calculus makes it possible to evaluate the limiting normal to the graph of $\Gamma$ as in \eqref{mor2-v} and \eqref{euler1} (which relates to the {\em coderivative} \cite{mo06a} of $\Gamma$) and represent it in terms of the problem data in structural problems. We refer the reader to \cite[Section~4.3]{mo06} for calculations of the coderivative for various classes of set-valued mappings $\Gamma$ and to \cite{mnp12,mny09} for their particular applications to KKT systems in nonlinear programming.\vspace*{0.03in}

Finally in this paper, we evaluate the limiting subdifferential of the random optimal value function $v(\cdot,\omega)$ for the stochastic DP problem \eqref{P} entirely via the problem data for the case of the random constraint mapping $\Gamma(x,\omega)$ defined in the nonlinear programming (NLP) form by
\begin{equation}\label{eq4}
\Gamma(x,\omega):=\left\{y\in X\Bigl|\;\begin{array}{ll}\varphi_i(x,y,\omega)\le 0,\;i=1,\ldots,m,\\
\varphi_i(x,y,\omega)=0,\;i=m+1,\ldots,m+r
\end{array}\right\},
\end{equation}
where the functions $\varphi_i:X\times X\times\Omega\to\R$, $i=1,\ldots,m+r$, are such that $\ph_i(\cdot,\cdot,\omega)$ are strictly differentiable at the reference point $(\ox,\oy)$ with their strict derivatives $\nabla\ph_i(\ox,\oy,\omega)$ while $\ph_i(x,y,\cdot)$ are measurable in accordance with assumption (H2) on $\Gamma$. Let us first recall the classical constraint qualification in nonlinear programming. Given $(\ox,\oy)\in X\times X$, denote the collection of {\em active constraint indices} by
$$
I(\bar{x},\bar{y},\omega):=\big\{i\in\{ 1,\ldots,m+r\}\big|\;\varphi_i(\bar{x},\bar{y},\omega)=0\big\}.
$$
It is said that the parametric {\em Mangasarian-Fromovitz constraint qualification} (MFCQ) holds at $(\ox,\oy)$ if for every $\omega\in\Omega$ we have the conditions:\vspace*{0.03in}

{\bf(i)} The equality constraint gradients $\nabla\varphi_{m+1}(\bar{x},\bar{y},\omega),\ldots,\nabla\varphi_{m+r}(\bar{x},\bar{y},\omega)$ are linearly independent in $X\times X$.

{\bf(ii)} There exists $\xi\in X\times X$ such that $\langle\nabla\varphi_i(\bar{x},\bar{y},\omega),\xi\rangle=0$ for $i=m+1,\ldots,m+r$ and
$\langle\nabla\varphi_i(\bar{x},\bar{y},\omega),\xi\rangle<0$ for $i\in\{1,\ldots,m\}\cap I(\bar{x},\bar{y},\omega)$.\vspace*{0.03in}

Denote the objective function $\varphi:X\times X\times\Omega\to\R$ in \eqref{P} by
\begin{eqnarray}\label{obj}
\varphi(x,y,\omega):=u(x,y,\omega)+\beta\disp\int v(y,\omega')P(d\omega'\mid\omega).
\end{eqnarray}
Taking $\lambda:=(\lambda_1,\ldots,\lambda_{m+r})\in\R^{m+r}$ and assuming that $u(\cdot,\cdot,\omega)$ is strictly differentiable at $(\ox,\oy)$ as $\omega\in\O$, consider the {\em set of Lagrange multipliers}
$$
\Lambda(\bar{x},\bar{y},\omega):=\left\{\lambda\in \R^{m+r}\Biggl|\;\begin{array}{ll}
\nabla_y\varphi(\bar{x},\bar{y},\omega)+\disp\sum_{i=1}^{m+r}\lambda_i\nabla_y\varphi_i(\bar{x},\bar{y},\omega)=0,\\
\lambda_i\ge 0,\;\lambda_i\varphi_i(\bar{x},\bar{y},\omega)=0,\;i=1,\ldots,m
\end{array}\right\},
$$
where we get by \eqref{obj} and the classical Leibniz rule that
$$
\nabla_y\varphi(\bar{x},\bar{y},\omega)=\nabla_y u(\bar{x},\bar{y},\omega)+\beta\disp\int\nabla v(\bar{y},\omega')P(d\omega'\mid\omega).
$$
Observe that the set $\Lambda(\bar{x},\bar{y},\omega)$ describes the Lagrangian stationary condition $\nabla_y L(\bar{x},\bar{y},\lambda,\omega)=0$ via the Lagrange function
$$
L(x,y,\lambda,\omega):=\varphi(x,y,\omega)+\sum_{i=1}^{m+r}\lambda_i\varphi_i(x,y,\omega).
$$

The following property is needed for the formulation of the next theorem. We say that a function $f:D\subset X\to X$ is \textit{locally upper Lipschitzian} at $\bar{x}\in D$ if there exist real numbers $\eta>0$ and $\ell\ge 0$ such that
$$
\|f(x)-f(\bar{x})\|\le\ell\|x-\bar{x}\|\quad\text{for every }\;x\in B_\eta(\bar{x})\cap D.
$$
Such a function is said to be a \textit{local upper Lipschitzian selector} of a given multifunction $F:D\rightrightarrows X$ around $(\bar{x},\bar{y})\in\mathrm{gph}\,F$ if $f(\bar{x})=\bar{y}$ and $f(x)\in F(x)$ for every $x$ in some neighborhood of $\bar{x}$.

Now we are ready to present the last result of the paper. Observe that its finite-dimensional version in the inclusion form can be found in \cite{rs09} for problems of type \eqref{P}  with only the inequality constraints in \eqref{eq4} given by smooth functions under the standard interiority condition and the linear independence of the gradients $\nabla_y\varphi_i(\bar{x},\bar{y},\omega)$, $i\in I(\bar{x},\bar{y},\omega)$, which is surely stronger than the parametric MSCQ imposed below.

\begin{Theorem}[\bf limiting subgradients of the value function in stochastic DP with NLP constraint mapping]\label{nlp}
Let $X$ be a separable Asplund space, and let the constraint multifunction $\Gamma:X\times\Omega\rightrightarrows X$ of \eqref{P} is given in the NLP form \eqref{eq4}. Assume that $u(\cdot,\cdot,\omega)$ and $\varphi_i(\cdot,\cdot,\omega)$, $i=1,\ldots,m+r$, are strictly differentiable at $(\bar{x},\bar{y})\in\mathrm{gph}\,G(\cdot,\omega)$, that the policy multifunction $G$ is inner semicontinuous at $(\bar{x},\bar{y})$, and that the parametric MFCQ is satisfied at this point. Then we have the inclusion
$$
\partial v(\bar{x},\omega)\subset\bigcup_{\lambda\in\Lambda(\bar{x},\bar{y},\omega)}\left\{\nabla_x u(\bar{x},\bar{y},\omega)+\sum_{i=1}^{m+r}\lambda_i\nabla_x\varphi_i(\bar{x},\bar{y},\omega)\right\},\quad\omega\in\O.
$$
Furthermore, the equality holds therein provided that $G(\cdot,\omega)$ admits a local upper Lipschitzian selector around the reference point $(\bar{x},\bar{y})$.
\end{Theorem}
\begin{proof} The Bellman equation \eqref{bell} ensures the representation of $v(x,\omega)$ as a marginal function \eqref{mf} under the separability assumption. It remains to use the subdifferential results from \cite[Corollary~4]{mny09}, which work in the general framework of \eqref{mf} in Asplund spaces and ensure the validity of both statements of the theorem under the imposed assumptions.\end{proof}\vspace*{0.05in}
{\bf Acknowledgements.} The authors are gratefully indebted to Mari\'an Fabian, Ali Khan, Lionel Thibault, and anonymous referee for their valuable remarks that allowed us to improve the original presentation.


\small
\begin{thebibliography}{00}
\bibitem{ab06}
C. D. Aliprantis and K. C. Border, {\em Infinite Dimensional Analysis: A Hitchhiker's Guide}, 3rd edn., Springer, Berlin, 2006.

\bibitem{acr09}
C. D. Aliprantis, G. Camera and F. Ruscitti, Monetary equilibrium and the differentiability of the value function, {\em J. Econom. Dynam. Control} {\bf 33} (2009), 454--462.

\bibitem{am96}
R. Amir, Sensitivity analysis of multisector optimal economic dynamics, {\em J. Math. Econom.} {\bf 25} (1996), 123--141.

\bibitem{alv98}
K. Askri and C. Le Van, Differentiability of the value function of nonclassical optimal growth models, {\em J. Optim. Theory Appl.} {\bf 97} (1998), 591--604.

\bibitem{bs79}
L. M. Benveniste and J. A. Scheinkman, On the differentiability of the value function in dynamic models of economics, {\em Econometrica} {\bf 47} (1979), 727--732.

\bibitem{blv96}
J. M. Bonnisseau and C. Le Van, On the subdifferential of the value function in economic optimization problems, {J. Math. Econom.} {\bf 25} (1996), 55--73.

\bibitem{ckr10}
B. Cascales, V. Kadets and J. Rodr\'iguez, Measurability and selections of multi-\hspace{0pt}functions in Banach spaces, {\em J. Convex Anal.} {\bf 17} (2010), 229--240.

\bibitem{ckr11}
B. Cascales, V. Kadets and J. Rodr\'iguez, The Gelfand integral for multi-\hspace{0pt}valued functions, {\em J. Convex Anal.} {\bf 18} (2011), 873--895.

\bibitem{ch09}
N. H. Chieu, The Fr\'{e}chet and limiting subdifferentials of integral functionals on the spaces $L_1(\Omega,E)$, {\em J. Math. Anal. Appl.} {\bf 360} (2009), 704--710.

\bibitem{cl83}
F. H. Clarke, {\em Optimization and Nonsmooth Analysis}, John Wiley \& Sons, New York, 1983.

\bibitem{cf16}
M. C\'uth and M. Fabian, Asplund spaces characterized by rich families and separable reduction of Fr\'echet subdifferentiability, {\em J.  Funct. Anal.} {\bf 270} (2016), 1361--1378.

\bibitem{da14}
A. Dasgupta, {\em Set Theory: With an Introduction to Real Point Sets}, Birkh\"auser, New York, NY, 2014.

\bibitem{du77}
J. Diestel and J. J. Uhl, Jr., {\em Vector Measures}, Amer. Math. Soc., Providence, RI, 1977.

\bibitem{fm02} M. Fabian and B. S. Mordukhovich, Separable reduction and extremal principles in variational analysis, {\em Nonlinear Anal.} {\bf 49} (2002), 265--292.

\bibitem{fk02}
S. Fajardo and H. J. Keisler, {\em Model Theory of Stochastic Processes}, A. K. Peters, Ltd., Natick, MA, 2002.

\bibitem{fs93} W. H. Fleming and H. M. Soner, {\em Controlled Markov Processes and Viscosity Solutions}, Springer, New York, 1993.

\bibitem{fr08}
D. Fremlin, {\em Measure Theory, Vol.\,5: Set-\hspace{0pt}Theoretic Measure Theory, Part I}, Torres Fremlin, Cochester, 2008.

\bibitem{gi08}
E. Giner, Subdifferential regularity and characterizations of the Clarke subgradients of integral functionals, {\em J. Nonlinear Convex Anal.} {\bf 9} (2008), 25--36.

\bibitem{gp15}
M. Greinecker and K. Podczeck, An exact Fatou's lemma for Gelfand integrals by means of Young measure theory, to appear in {\em J. Convex Anal.}

\bibitem{hk84}
D. Hoover and H. J. Keisler, Adapted probability distributions, {\em Trans. Amer. Math. Soc.} {\bf 286} (1984), 159--201.

\bibitem{il72}
A. D. Ioffe and V. L. Levin, Subdifferentials of convex functions, {\em Trans. Moscow Math. Soc.} {\bf 26} (1972), 1--72.

\bibitem{jt11}
A. Jourani and L. Thibault, Noncoincidence of approximate and limiting subdifferentials of integral functionals, {\em SIAM J. Control Optim.} {\bf 49} (2011), 1435--1453.

\bibitem{ks09}
H. J. Keisler and Y. N. Sun, Why saturated probability spaces are necessary, {\em Adv. Math.} {\bf 221} (2009), 1584--1607.

\bibitem{kh85}
M. A. Khan, On the integration of set-\hspace{0pt}valued mappings in a non-\hspace{0pt}reflexive Banach space, II, {\em Simon Stevin} {\bf 59} (1985), 257--267.

\bibitem{ks13}
M. A. Khan and N. Sagara, Maharam-\hspace{0pt}types and Lyapunov's theorem for vector measures on Banach spaces, {\em Illinois J.  Math.} {\bf 57} (2013), 145--169.

\bibitem{ks15a}
M. A. Khan and N. Sagara, Maharam-\hspace{0pt}types and Lyapunov's theorem for vector measures on locally convex spaces with control measures, {\em J. Convex Anal.} {\bf 22} (2015), 647--672.

\bibitem{ki93}
T. Kim, Differentiability of the value function: A new characterization, {\em Seoul J. Econom.} {\bf 6} (1995), 257--265.

\bibitem{ma42}
D. Maharam, On homogeneous measure algebras, {\em Proc. Natl. Acad. Sci. USA} {\bf 28} (1942), 108--111.

\bibitem{mau01}
A. Maugeri, Equilibrium problems and variational inequalities, in {\em Equilibrium Problems: Nonsmooth Optimization and Variational Inequality Methods} (F. Giannessi et al., eds.), {\em Nonconvex Optim. Appl.} {\bf 58}, pp.\ 187--205, Kluwer, Dordrecht, 2001.

\bibitem{mo06a}
B. S. Mordukhovich, {\em Variational Analysis and Generalized Differentiation, I: Basic Theory}, Springer, Berlin, 2006.

\bibitem{mo06}
B. S. Mordukhovich, {\em Variational Analysis and Generalized Differentiation, II: Applications}, Springer, Berlin, 2006.

\bibitem{mn05}
B. S. Mordukhovich and N. M. Nam, Variational stability of marginal functions via generalized differentiation, {\em Math. Oper.
Res.} {\bf 30} (2005), 800--816.

\bibitem{mnp12}
B. S. Mordukhovich, N. M. Nam and H. M. Phan, Variational analysis of marginal functions with applications to bilevel programming, {\em J. Optim. Theory Appl.} {\bf 152} (2012), 557--586.

\bibitem{mny09}
B. S. Mordukhovich, N. M. Nam and N. D. Yen, Subgradients of marginal functions in parametric mathematical programming, {\em Math. Program.} {\bf 116} (2009), 369--396.

\bibitem{mu13}
K. Musial, Approximation of Pettis integrable multifunctions with values in arbitrary Banach spaces, {\em J. Convex Anal.} {\bf 20} (2013), 833--870.

\bibitem{pe11}
J. P. Penot, Image space approach and subdifferentials of integral functionals, {\em Optimization} {\bf 60} (2011), 69--87.

\bibitem{ph93}
R. R. Phelps, {\em Convex Functions, Monotone Operators and Differentiability}, 2nd edn., Springer, Berlin, 1993.

\bibitem{po08}
K. Podczeck, On the convexity and compactness of the integral of a Banach space valued correspondence, {\em J. Math. Econom.} {\bf 44} (2008), 836--852.

\bibitem{rs09}
J. P. Rinc\'on-\hspace{0pt}Zapatero and M. S. Santos, Differentiability of the value function without interiority assumptions, {\em J. Econom. Theory} {\bf 144} (2009), 1948--1964.

\bibitem{ro71}
R. T. Rockafellar, Convex integral functionals and duality, in {\em Contributions to Nonlinear Functional Analysis} (E. H. Zarantonello, ed.), pp.\ 215--236, Academic Press, New York, 1971.

\bibitem{sc73}
L. Schwartz, {\em Radon Measures on Arbitrary Topological Spaces and Cylindrical Measures}, Oxford University Press, London, 1973.

\bibitem{slp89}
N. L. Stokey, R. E. Lucas, Jr. and E. C. Prescott, {\em Recursive Methods in Economic Dynamics}, Cambridge, Massachusetts, Harvard University Press, 1989.

\bibitem{sy08}
Y. N. Sun and N. C. Yannelis, Saturation and the integration of Banach valued correspondences, \textit{J.\ Math.\ Econom.}\ \textbf{44} (2008), 861--865.

\bibitem{t91} L. Thibault, On subdifferentials of optimal value functions, {\em SIAM J. Control Optim.} {\bf 29} (1991), 1019--1036.

\bibitem{th75}
G. E. F. Thomas, Integration of functions with values in locally convex Suslin spaces, {\em Trans. Amer. Math. Soc.} {\bf 212} (1975), 61--81.

\bibitem{va75}
M. Valadier, Convex integrands on Souslin locally convex spaces, {\em Pacif. J. Math.} {\bf 59} (1975), 267--276.
\end{thebibliography}
\end{document}